\documentclass[aap,MSNbibl,seceqn,citesort,dvips]{arximspdf}
\usepackage{mathbh}
\usepackage{graphicx}
\doi{10.1214/11-AAP782}
\volume{22}
\issue{2}
\pubyear{2012}
\firstpage{670}
\lastpage{701}

\makeatletter
\newtheorem{teo}{Theorem}[section]

\newtheorem{teorema}{Theorem}[section]

\newtheorem{proposizione}[teorema]{Proposition}
\newtheorem{lemma}[teorema]{Lemma}
\newtheorem{cor}[teorema]{Corollary}
\newproclaim{remark}[teorema]{Remark}
\newproclaim{definition}[teorema]{Definition}

\newcommand{\I}{\mathbh{1}}
\newcommand{\OmNless}{\Omega_N}
\newcommand{\OmN}{\overline{\Omega}_N}
\newcommand{\Dphi}{{\widetilde{\phi}}}
\newcommand{\CNS}{2.4}
\newcommand{\PRM}{2.7}
\newcommand{\Cor}{3.28}
\makeatother

\begin{document}
\begin{frontmatter}

\title{On the role of Allee effect and mass migration in survival and extinction of a species\thanksref{T1,T2}}
\thankstext{T1}{Supported by Fondation Sciences Math\'ematiques de Paris.}
\thankstext{T2}{Supported by the French Ministry of Education through
the Grant ANR BLAN07-218426.}
\runtitle{Critical parameters for metapopulation models}

\begin{aug}
\author{\fnms{Davide} \snm{Borrello}\corref{}\ead[label=e1]{d.borrello@campus.unimib.it}}
\runauthor{D. Borrello}
\affiliation{Universit\`a degli Studi di Milano Bicocca and
CNRS--Universit\'e de Rouen}
\address{Dipartimento di Matematica e Applicazioni\\
Universit\`a degli Studi di Milano Bicocca\\
Via Cozzi 53\\
20125 Milano\\
Italy\\
and\\
Laboratoire de Math\'ematiques Rapha\"el Salem\\
UMR 6085 CNRS--Universit\'e de Rouen\\
Avenue de l'Universit\'e BP.12\\
F76801 Saint-\'Etienne-du-Rouvray\\
France\\
\printead{e1}} 
\end{aug}

\received{\smonth{3} \syear{2010}}
\revised{\smonth{4} \syear{2011}}

%
\begin{abstract}
We use interacting particle systems to investigate survival and
extinction of a species with colonies located on each site of $\mathbb
{Z}^d$. In each of the four models studied, an individual in a local
population can reproduce, die or migrate to neighboring sites.

We prove that an increase of the death rate when the local population
density is small (the Allee effect) may be critical for survival, and
that the migration of large flocks of individuals is a possible
solution to avoid extinction when the Allee effect is strong. We use
attractiveness and comparison with oriented percolation, either to
prove the extinction of the species, or to construct nontrivial
invariant measures for each model.
\end{abstract}

\begin{keyword}[class=AMS]
\kwd[Primary ]{60K35}
\kwd{60K35}
\kwd[; secondary ]{82C22}.
\end{keyword}

\begin{keyword}
\kwd{Interacting particle systems}
\kwd{phase transition}
\kwd{metapopulation models}
\kwd{Allee effect}
\kwd{mass migration}
\kwd{stochastic order}
\kwd{comparison with percolation}.
\end{keyword}


\end{frontmatter}

\section{Introduction}
\label{secintrmetapopulation}

A metapopulation model refers to many small local populations connected
via migrations in a fragmented environment. Each local population
evolves without spatial structure; it can increase or decrease,
survive, get extinct or migrate from its site in different ways; see
\cite{cfHanski} for more about metapopulations.

The most natural model for the evolution of a single population is the
branching process; see~\cite{cfGaltonWatts}: birth and death rates
depend on the number of individuals of the population, and the growth
rate is density dependent.

If the birth rate is always larger than the death rate, if the
population survives, it will increase indefinitely. If the birth rate
is smaller than or equal to the death rate, the population will become
extinct almost surely~\cite{cfWilliams}. A more interesting
situation is given by a birth rate larger than the death rate under a
particular population size~$N$, and smaller over that. The real
environments observation suggests that this process is gradual; that
is, the growth rate decreases over a population size as population
density increases. In some of our applications we suppose that over a
fixed number of individuals $N$ (the \textit{capacity} of a site), the
growth rate is null.

Many biological phenomena may influence the dynamics of a
metapopulation.

Migration is one of the most important strategies that a species adopts
to improve its probability of survival (see, e.g., \cite
{cfBrassil,cfHanski,cfalleestephens}) when the
population size is large one or more individuals leave the site where
they are located to look for new resources in different sites.

Other biological factors may favor the extinction of a species. One of
them, \textit{the Allee effect}, consists of an increase in the death
rate when the density of individuals is small. The reason is that at
low density many factors (as difficulties in finding mates) cause a
decrease of fecundity and an increase of mortality;
see~\cite{cfallee1,Allee1,Allee4,cfalleestephens}.

We simplify the real structure, and we treat $4$ metapopulation models
from a mathematical point of view: we start from the easier one by
adding a new biological phenomenon at each model.

The mathematical models are interacting particle systems on $\Omega
=X^{\mathbb{Z}^d}$, where $X \subseteq\mathbb{N}$: each particle
represents one individual and on each site of~$\mathbb{Z}^d$ there is a
local population with capacity $N$ (possibly $N=\infty$), which evolves
in different ways depending on the model. The local populations are
connected via migrations of individuals, that is, jumps of particles
from a~site to another one.

In Section \ref{Background} we introduce the particle system, give the
main definitions and notation and state the attractiveness results,
crucial in the sequel for the existence of critical parameters and
nontrivial invariant measures. Theorem~\ref{cns}, the main result of a
previous paper (\cite{cfBorrello}, Theorem \CNS, inspired by~\cite
{GobronSaada}), gives necessary and sufficient conditions for
attractiveness of a~large class of particle systems. This simplifies
many proofs, since, in order to derive either, if two processes are
stochastically ordered, or if a process is attractive, we do not need
to construct an explicit coupling for each model, but we only have to
check inequalities involving the transition rates.

In \cite{cfschiallee} and \cite{cfschiaggr}, the author considers a
metapopulation model to investigate the roles of mass death (i.e., the
death of all individuals in a local population) and spatial aggregation
in the extinction of a species. In \cite{cfschiallee} he shows that,
in presence of mass death, animals living in large flocks are more
susceptible to extinction than animals living in small flocks: for this
model, mass death can be an alternative to the Allee effect in raising
to the extinction of a species. The new results in~\cite{cfschiaggr}
involve the role of spatial aggregation, which may be either bad or
good for survival in a model respectively, with or without mass death.
For these models the local population Allee effect was not taken into
account. The model introduced, called a \textit{noncatastrophic times
model}, is the following: for a fixed $N<\infty$, on each site of
$\mathbb{Z}^{d}$ we may have up to $N$ individuals; hence $N$ is the
\textit{capacity} of sites. The transitions of the Markov process $(\eta
_{t})_{t\geq0}$ are
\begin{eqnarray*}
\eta_t(x) &\to&\eta_t(x)+1 \qquad\mbox{at rate } \eta_t(x)\varphi+\lambda
\sum_{y\sim x}\I_{\{\eta_t(y)=N\}} \mbox{ for }0\leq\eta_t(x)< N,\\
\eta_t(x)&\to&\eta_t(x)-1 \qquad\mbox{at rate } 1 \mbox{ for }1\leq\eta
_t(x)\leq N,
\end{eqnarray*}
where $y \sim x$ are neighbors. In other words, each individual gives
birth to another one on the same site with rate $\varphi$ and dies with
rate $1$. An individual on site $x$ gives birth to a new individual in
a neighboring site with rate $\lambda/N$ only when the population at
$x$ has reached the maximal size $N$. There is a~critical parameter for
the capacity $N$ of sites:
\begin{teo}[({\cite{cfschiaggr}, Theorem 2})]
Assume that $d\geq2$, $\lambda>0$ and $\varphi>0$. There is a critical
value $N_{c}(\lambda,\varphi)$ such that if $N>N_{c}(\lambda,\varphi)$,
then starting from any finite number of individuals, the population has
a strictly positive probability of surviving.
\label{intrschi1}
\end{teo}

Starting from noncatastrophic times model, we propose $4$ models to
improve the understanding of species dynamics.
We want to investigate, \textit{for the first time in a model with
spatial structure}, the role of the Allee effect, the role of mass
migration and their interactions.

In Section \ref{ModelI} we introduce \textit{Model I}. This will
represent our basic model with neither Allee effect nor mass migration.
We begin with a system very similar to Schinazi's model: since a
further step consists in adding migration of many individuals, we
consider a migration of one individual to a neighboring site instead of
a birth of a new individual. If $N=1$, such a difference does not allow
survival for the model with migrations, since no new births are
possible, and the process gets extinct for any $\lambda$: this is
definitely not the case for the noncatastrophic times model with
$N=1$, which is the contact process. If $N$ is large this small
difference does not change the behavior of the model.

This is the basic model, and it must be as easy as possible (births and
deaths on the same site and migrations from one site to another, all
for at most one particle at time). For this reason we do not consider
mass death, which is an additional complex factor.

We take the birth rate larger than the death rate, but we fix a
capacity~$N$ per site. A migration of one individual from a site $x$
toward a nearest neighbor one, is allowed only when the population on
$x$ reaches $N$. We prove that in some cases there is almost sure
extinction, and in others the species survives with positive
probability: the key tool to prove survival is the comparison technique
with a supercritical oriented percolation model; see
\cite{cfdurrettten}.\vadjust{\goodbreak}

In Section \ref{ModelII} we introduce \textit{Model II}, that points out
the key role of the Allee effect in species dynamics. Schinazi used
mass death to prove that it can be considered an alternative to the
Alle effect for extinction of a~species. Since both the Allee effect
and mass death improve the probability of extinction, in order to
understand the role of one of them they should be considered separately.
Here we want to show that a strong Allee effect (with neither mass
death nor mass migration) is a key factor for the extinction.

We add the Allee effect to Model I. Different probabilistic tools have
already been used to illustrate the Allee effect, like stochastic
differential equations (see~\cite{Dennis1}), discrete-time Markov
chains (see \cite{Allen}) or diffusion processes (see~\cite{Dennis2}),
but none of these models has a spatial structure.

In Model II each site has a capacity $N$, but the death rate is larger
than the birth rate for small densities. Migration works exactly as in
Model I. Theorem \ref{phasetransition2} states that for all possible
capacities, growth and migration rates, there exists an Allee effect
large enough for the species to become extinct. It is proved through
comparison with subcritical percolation.

In \textit{Model III}, introduced in Section \ref{ModelIII}, we allow a
migration of more than one individual at a time from one site to the
neighboring one in a species affected by the Allee effect. We prove
that \textit{mass migration might be the possible strategy of a species
to reduce the Allee effect} and improve its survival probability.

When a local population size reaches $N$, a migration of a number of
individual smaller than a fixed $M$ is possible.
In Model II, for an Allee effect large enough, the species gets
extinct. In Model III, if $N$ is large enough there exists $M$ such
that this is no longer true. A migration of large flocks avoids small
densities in a new environment which are bad for survival. Indeed, by
comparison arguments with oriented percolation, even if the Allee
effect is the strongest one, if the species lives and migrates in
flocks large enough, survival is possible (Theorem \ref{phasetransition3}).

In Section \ref{ModelIV} we generalize the previous models: in \textit
{Model IV}, instead of fixing a capacity $N$, we consider a slightly
more realistic model. In all environments there is no maximal size, but
a kind of self-mechanism of birth control such that the death rate is
larger than the birth rate when there are more than $N$ individuals in
a local population. A migration of one or more individuals is allowed
from a site with more than $N$ individuals toward a~site with few
individuals. We prove in Theorem~\ref{phasetransition4} that in some
cases we can have survival but on each site the population does not
explode even if there is no capacity. Namely, on each site the expected
value of the number of individuals is finite. In other cases the
species becomes extinct.

Note that on each model instead of fixing the death rate equal to 1 and
letting the birth rate vary (the most used approach), we consider the
reverse but equivalent point of view in order to clarify our proofs,
presented in
Section \ref{proofs}.

\section{Background and tools}
\label{Background}

The mathematical model is an interacting particle system $(\eta
_{t})_{t\geq0}$ on $\Omega=X^{\mathbb{Z}^d}$, where $X=\{0,1,\ldots,N\}
\subseteq\mathbb{N}$ and $N$ denotes the common size (capacity) of the
local populations, if finite. The value~$\eta_{t}(x)$, $x \in\mathbb
{Z}^{d}$, is the number of individuals present in site $x$ at time $t
\geq0$. We write~$\OmNless$ when we want to stress the dependency on
the capacity $N$.

When $X$ is finite, which is the case of Models I, II and III, we refer
to the construction in \cite{cfLiggett}; when $X$ is infinite, that
is, in Model IV, the state space is noncompact, and a different
construction is needed. The first examples of interacting particle
systems with locally interacting components in noncompact state spaces
have been introduced in \cite{Spitzer}. One approach to construct these
kinds of models has been developed in \cite{cfLigSpitz}, where the
construction was detailed for Coupled Random Walks, but with small
changes it can be generalized to many other processes. By using similar
ideas, in \cite{cfChenbook} was stated a general existence theorem for
reaction-diffusion processes, that we are going to apply in Model IV:
in order to assure the existence of the process, some restrictions on
the transition rates are required, as explained in Section \ref
{ModelIV}.

The process admits an \textit{invariant measure} $\mu$ if $P_{\mu}(\eta
_{t}\in A)=\mu(A)$ for each $t\geq0$, $A \subseteq\Omega$, where
$P_{\mu}$ is the law of the process with initial distribution $\mu$. An
invariant measure is \textit{trivial} if it is concentrated on an
absorbing state, when one exists. The process is \textit{ergodic} if
there is a unique invariant measure to which the process converges
starting from each initial distribution (see~\cite{cfLiggett}, Definition
1.9). For any $x,y \in\mathbb{Z}^{d}$, we write $y \sim
x$ if $y$ is one of the $2d$ nearest neighbors of site $x$.

We introduce here a common infinitesimal generator $\mathcal{L}$ (we
will be more precise on each model): it is given by
\begin{eqnarray}\label{generator}\qquad
\mathcal{L}f(\eta)&=&\sum_{x \in\mathbb{Z}^d}\sum_{k \in X}\biggl\{
P^k_{\eta(x)}\bigl(f(S^{k}_{x}\eta)-f(\eta) \bigr)+P^{-k}_{\eta(x)}
\bigl(f(S^{-k}_{x}\eta(x))-f(\eta)\bigr)
\nonumber
\\[-8pt]
\\[-8pt]
\nonumber
&&\hspace*{98pt}{}+\sum_{y \sim x}\frac{1}{2d}\Gamma^k_{\eta(x),\eta(y)}
\bigl(f(S^{-k,k}_{x,y}\eta)-f(\eta)\bigr)\biggr\},
\end{eqnarray}
where $f$ is a local function, $\eta\in\Omega$, $S^{-k,k}_{x,y}$,
$S^{k}_{y}$ and $S^{-k}_{y}$, where $k > 0$, are local operators
performing the transformations whenever possible
\begin{eqnarray}
\label{intrSkk}
(S^{-k,k}_{x,y}\eta)(z)&=&\cases{
\eta(x)-k, & \quad$\mbox{if } z=x \mbox{ and }\eta(x)-k \in X, \eta(y)+k \in
X,$\vspace*{2pt}\cr
\eta(y)+k, & \quad$\mbox{if } z=y \mbox{ and }\eta(x)-k \in X, \eta(y)+k \in
X,$\vspace*{2pt}\cr
\eta(z), &\quad$\mbox{otherwise},$}\hspace*{-35pt}
\\
\label{intrSk}(S^{k}_{y}\eta)(z)&=&\cases{
\eta(y)+k, & \quad $\mbox{if } z=y \mbox{ and } \eta(y)+k \in
X,$\vspace*{2pt}\cr
\eta(z), & \quad$\mbox{otherwise},$}
\\
\label{intrS-k}
(S^{-k}_{y}\eta)(z)&=&\cases{
\eta(y)-k, & \quad $\mbox{if } z=y \mbox{ and } \eta(y)-k \in
X,$\vspace*{2pt}\cr
\eta(z), & \quad$\mbox{otherwise},$}
\end{eqnarray}
$P^{k}_{\cdot}$, $P^{-k}_{\cdot}$ are positive functions from $X$ to
$\mathbb{R}$, and in our four models $k=0,1$ (particles are born and
die one at a time).

\textit{We assume $P^1_0=0$, that is, the Dirac measure concentrated on
the empty configuration $\delta_{\underline{0}}$ is a trivial invariant
measure}. The function\vspace*{1pt} $\Gamma^k_{\eta(x),\eta(y)}$ represents the
migration (jump) rate; a jump of more than one particle per time is
possible. We call \textit{emigration} from $x$ a jump that reduces the
number of particles on $x$ and \textit{immigration} a jump that increases
it.

There is a natural definition of partial order on the state space,
%
\begin{equation}
\forall\xi,\eta\in\Omega,\qquad \xi\leq\eta\quad\Leftrightarrow\quad
\bigl(\forall x \in S, \xi(x) \leq\eta(x)\bigr).
\end{equation}
A process $(\eta_{t})_{t\geq0}$ with\vspace*{1.5pt} generator $\mathcal{L}$ is \textit
{stochastically larger} than a process $(\xi_{t})_{t\geq0}$ with
generator $\widetilde{\mathcal{L}}$ if, given $\xi_0\leq\eta_0$, there
exists an increasing Markovian coupling $(\xi_t,\eta_t)_{t\geq0}$ on
state space $\Omega\times\Omega$ such that
\[
\mathbb{P}^{(\xi_0,\eta_0)}(\xi_t\leq\eta_t)=1,
\]
for all $t\geq0$, where $\mathbb{P}^{(\xi_0,\eta_0)}$ denotes the
distribution of $(\xi_t,\eta_t)_{t\geq0}$ with initial state $(\xi
_0,\eta_0)$.\vspace*{1pt} In this case the process $(\xi_{t})_{t\geq0}$ is
stochastically smaller than $(\eta_{t})_{t\geq0}$, and the pair $(\xi
_{t},\eta_{t})_{t\geq0}$ is stochastically ordered; see \cite{cfBorrello}, Section
\ref{Background}. If $\mathcal{L}=\widetilde{\mathcal{L}}$, and there is
stochastic order between two processes with ordered initial
configurations, then the process is \textit{attractive}; see
\cite{cfLiggett},
Definition II.2.2.

Necessary and sufficient conditions for stochastic order and
attractiveness in a general class of particle systems including the
models defined by generator (\ref{generator}) have been derived by
\cite{cfBorrello}, Theorem \CNS, which generalizes \cite{GobronSaada}, Theorem
2.21. Since (\ref{generator}) involves neither births nor
deaths depending on neighboring sites, this theorem can be restated as follows:
\begin{teorema}[({\cite{cfBorrello}, Theorem \CNS})]
Given $K\in\mathbb{N}$, $\mathbf{j}:=\{j_{i}\}_{1\leq i \leq K}$,
$\mathbf{m}:=\{m_{i}\}_{1\leq i \leq K}$, $\mathbf{h}:=\{h_{i}\}_{1\leq
i \leq K}$, three nondecreasing $K$-uples in $\mathbb{N}$, and $\alpha
,\beta,\gamma,\delta$ in~$X$ such that $\alpha\le\gamma$, $\beta\leq
\delta$, we define
%
\begin{eqnarray}
I_{a}&:=&I^K_{a}(\mathbf{j},\mathbf{m})= \bigcup^{K}_{i=1}\{
k\in X\dvtx m_{i}\geq k > \delta-\beta+j_{i}\},
\label{Ia}\\
I_{b}&:=&I^K_{b}(\mathbf{j},\mathbf{m})= \bigcup^{K}_{i=1}\{
k\in X\dvtx \gamma-\alpha+m_{i} \geq k > j_{i}\},
\label{Ib}\\
I_{c}&:=&I^K_{c}(\mathbf{h},\mathbf{m})= \bigcup^{K}_{i=1}\{
k\in X\dvtx m_{i}\geq k > \gamma-\alpha+h_{i}\},
\label{Ic}\\
I_{d}&:=&I^K_{d}(\mathbf{h},\mathbf{m})= \bigcup^{K}_{i=1}\{
k\in X\dvtx \delta-\beta+m_{i} \geq k > h_{i}\}.
\label{Id}
\end{eqnarray}
A particle system\vspace*{-1pt} $(\eta_{t})_{t \geq0}$ with transition rates $\{
\Gamma^{k}_{a,b},P^{k}_{b},P^{-k}_{a}\}_{\{a,b,k\in X\}}$ is \textit
{stochastically larger} than a particle system $(\xi_{t})_{t \geq0}$
with transition rates
$\{\widetilde{\Gamma}^{k}_{a,b},\allowbreak\widetilde
{P}^{k}_{b}, \widetilde{P}^{-k}_{a}\}_{\{a,b,k\in X\}}$ if and only if
\begin{eqnarray}
\sum_{k\in X\dvtx  k > \delta-\beta+j_{1}}\widetilde{P}_{\beta}^{k}+\sum_{k
\in I_{a}}\widetilde{\Gamma}_{\alpha, \beta}^{k}&\leq&
\sum_{l\in X\dvtx l>j_{1}}P_{ \delta}^{l}+\sum_{l \in I_{b}}\Gamma_{\gamma,
\delta}^{l},
\label{C+}\\[-1pt]
\label{C-}
\sum_{k\in X\dvtx k>h_{1}}\widetilde{P}_{\alpha}^{-k}+\sum_{k \in
I_{d}}\widetilde{\Gamma}_{\alpha, \beta}^{k}&\geq&
\sum_{l\in X\dvtx  l > \gamma-\alpha+h_{1}}P_{\gamma}^{-l}+\sum_{l \in
I_{c}}\Gamma_{\gamma, \delta}^{l}
\end{eqnarray}
for all choices of $K$, $\mathbf{h}$, $\mathbf{j}$, $\mathbf{m}$,
$\alpha\leq\gamma$ and $\beta\leq\delta$.
\label{cns}
\end{teorema}
\begin{remark}
It is not possible that an infinite value for $K$, $I_a$, $I_b$, $I_c$,
$I_d$ results in the same rate inequality: therefore one restricts to
take $K$ smaller than the maximal change (birth, death or migration) of
particles involved in a transition; see \cite{cfBorrello}, Remark 2.5.
\end{remark}
\begin{remark}
To prove Theorem \ref{cns}, following the approach of \cite
{GobronSaada}, we first show that conditions (\ref{C+})--(\ref{C-}) are
necessary. Then we construct a~Markovian coupling which turns out to be
increasing under (\ref{C+})--(\ref{C-}); see \cite{cfBorrello}, Section~\ref{ModelI}.
Hence if conditions (\ref{C+})--(\ref{C-}) are not
satisfied it is not possible to find a coupling that preserves the
order between the two processes.
\label{cnsproof}
\end{remark}

By taking two processes with the same transition rates, Theorem \ref
{cns} states necessary and sufficient conditions for attractiveness. We
use attractiveness of a process to construct a nontrivial invariant
measure starting from an initial configuration $\eta_0 \in\OmN$, where
%
\begin{equation}
\OmN:=\{\eta\in\Omega\dvtx  \eta(x)=N \mbox{ for all }x \in\mathbb{Z}^d\}.
\label{Omegan}
\end{equation}

\begin{remark}[({\cite{cfBorrello}, Proposition \PRM})]\label{remM1} For processes
with births, deaths and jumps of at most one particle per site,
conditions (\ref{C+}) and (\ref{C-}) reduce to
\begin{eqnarray}
\widetilde{P}_{\beta}^{1}+\widetilde{\Gamma}_{\alpha, \beta}^{1}&\leq&
P_{ \delta}^{1}+\Gamma_{\gamma, \delta}^{1} \qquad \mbox{if } \beta=\delta
\mbox{ and } \gamma\geq\alpha,
\label{C+1}\\
\widetilde{P}_{\beta}^{1}&\leq& P_{\delta}^{1} \qquad \mbox{if } \beta
=\delta\mbox{ and } \gamma= \alpha,
\label{C+10}\\
\widetilde{P}_{\alpha}^{-1}+\widetilde{\Gamma}_{\alpha, \beta}^{1}&\geq
&P_{\gamma}^{-1}+\Gamma_{\gamma, \delta}^{1} \qquad \mbox{if } \gamma
=\alpha\mbox{ and } \delta\geq\beta,
\label{C-1}\\
\widetilde{P}_{\alpha}^{-1}&\geq&P_{\gamma}^{-1} \qquad \mbox{if } \gamma
=\alpha\mbox{ and } \delta= \beta.
\label{C-10}
\end{eqnarray}
\end{remark}
\begin{remark}
By \cite{cfBorrello}, Corollary \Cor, the sufficient condition still
holds if we consider systems with more\vspace*{1pt} general transition rates $\Gamma
^k_{\eta(x),\eta(y)}(x,y)$ and~$P^{k}_{\eta(x)}(x)$, not translation invariant.
In this case there is stochastic order if conditions (\ref{C+})--(\ref
{C-}) [resp., (\ref{C+1})--(\ref{C-10}) if $N=1$] are satisfied for each
pair of sites $(x,y)$ and configurations $\eta\leq\xi$ with $\eta
(x)=\alpha$, $\eta(y)=\beta$, $\xi(x)=\gamma$, $\xi(y)=\delta$.
\label{noninv}
\end{remark}

Remark \ref{noninv} will be used in some steps of the further proofs
(for Theorems~\ref{phasetransition1} and \ref{phasetransition2}), where
in order to make a comparison with oriented percolation, we will
introduce systems with different transition rates in different space
regions, so that they do not satisfy the hypothesis of Theorem~\ref{cns}.
\begin{definition}
For a process $(\eta_t)_{t \geq0}$ there is \textit{survival of the
species} if
%
\begin{equation}
\mathbb{P}(|\eta_{t}| \geq1 \mbox{ for all }t \geq0)>0,
\end{equation}
where $|\eta_{t}|$ denotes the number of individuals at time $t$, and
$|\eta_0|$ is finite. Otherwise the species \textit{becomes extinct}. If
the process starts from an infinite~$\eta_0$ we say
that the species becomes extinct if the process converges to $\delta
_{\underline{0}}$.
The convergence to $\delta_{\underline{0}}$ is intended that for any
finite $S \subset\mathbb{Z}^d$, the probability that there exists
$t_0$ such that for all $t>t_0$, $\eta_t(x)=0$ for all $x \in S$ tends
to 1.
\end{definition}

\section{Model I: The basic model}
\label{ModelI}

We introduce Model I. We choose to fix a~birth rate equal to $1$ and to
associate two parameters to death and migration rates.
Given~$\phi$ and $\lambda$ positive real numbers, transitions are, for
all $x \in S$, $y\in S$, $y \sim x$ [we follow the notation in (\ref{generator})]
\begin{eqnarray}\label{rates1}
\eta_t(x)&\to&\eta_t(x)+1 \qquad\mbox{at rate } P^1_{\eta_t(x)}=\eta_t(x)\I
_{\{\eta_t(x)<N\}},\nonumber\\
\eta_t(x)&\to&\eta_t(x)-1 \qquad \mbox{at rate } P^{-1}_{\eta_t(x)}=\phi
\eta_t(x),
\nonumber
\\[-8pt]
\\[-8pt]
(\eta_t(x),\eta_t(y))&\to&\bigl(\eta_t(x)-1,\eta_t(y)+1\bigr)\nonumber\\
\eqntext{\displaystyle \mbox{at rate }
\frac{1}{2d}\Gamma^1_{\eta_t(x),\eta_t(y)}=\frac{\lambda}{2d}\I_{\{\eta_t(x)=N,\eta_t(y)<N\}}.}
\end{eqnarray}
The model has the following monotonicity properties:
\begin{proposizione}
Let $(\xi_{t})_{t\geq0}$, $(\eta_{t})_{t\geq0}$ be two processes with
respective parameters $(\phi_{1},\lambda,N)$ and $(\phi_{2},\lambda
,N)$ such that $\phi_{1} \leq\phi_{2}$. Then $(\xi_{t})_{t\geq0}$ is
stochastically larger than $(\eta_{t})_{t \geq0}$, and $(\eta_{t})_{t
\geq0}$ is an attractive process.
\label{monotonicity0}
\end{proposizione}

The key for attractiveness, which is a consequence of the stochastic
ordering when $\phi_1=\phi_2$, is that there are births, deaths and
migrations of at most one particle per time and the migration rate from
$\eta_t(x)$ to $\eta_t(y)$ is nondecreasing in $\eta_t(x)$ and
nonincreasing in $\eta_t(y)$.
\begin{cor}\label{monotonicity1}
Given $(\eta^\xi_{t})_{t\geq0}$ such that $\eta^{\xi}_{0}=\xi$, then
\[
\mathbb{P}(|\eta^{\xi}_{t}| \geq1 \mbox{ for all }t \geq0)
\]
is nonincreasing in $\phi$ for each $\xi\in\Omega$.
\end{cor}

\begin{remark}
There is no stochastic order between systems with different values of
$N$ or $\lambda$. Indeed, in these cases, the conditions of Theorem \ref
{cns} are not satisfied.
\label{remNlambda}
\end{remark}

The first result corresponds to Theorem \ref{intrschi1} for the
noncatastrophic times model, and it is proved in a similar way.
\begin{teo}
Suppose $d\geq2$, $\lambda>0$ and $\phi<1$. There exists a critical
value $N_{c}(\lambda,\phi)$ such that if $N >N_{c}(\lambda,\phi)$, then
starting from $\eta_0\in\Omega$ such that $|\eta_0|\geq1$, the
process has a positive probability of survival. Moreover if $\eta_0 \in
\OmN$ the process converges to a nontrivial invariant measure with
positive probability.
\label{phasetransition1b}
\end{teo}
\begin{pf}
We skip the proof, since the result is a corollary of Theorem~\ref
{phasetransition3}.
We can get an easier proof that the process has a positive probability
of surviving by slightly modifying \cite{cfschiaggr}, proof of Theorem
2. The differences are that we consider a migration
instead of a birth from $x$ to $y\sim x$, and the migration rate from
$x$ to $y$ is nonincreasing in $\eta_{t}(y)$. Such changes are not
relevant for the proof.
\end{pf}

As we can expect, aggregation is good for Model I, as in
noncatastrophic times model.

\begin{remark}
If $N=1$ the process dies out, since each individual can only migrate
or die.

This suggests that an increase of $N$ is good for the survival of the
species. However, by Remark \ref{remNlambda}, there is no monotonicity
property with respect to~$N$.
\end{remark}

If we fix the capacity $N$, we prove that there is a phase transition
also with respect to the death rate $\phi$.
\begin{teo}\label{phasetransition1}
For all $\lambda>0$, $1<N<\infty$, there exists $\phi_{c}(\lambda,N)<
1$ such that, if $\phi< \phi_{c}(\lambda,N)$ the process starting from
$\eta_0$ with $1 \leq|\eta_0|< \infty$ has a positive probability of
survival and if $\phi>\phi_{c}(\lambda,N)$, the process dies out.
Moreover, for $\eta_0\in\OmN$ if $\phi< \phi_{c}(\lambda,N)$, the
process converges to a nontrivial invariant measure with positive probability.
\end{teo}

We prove it in three steps in Section \ref{firstproof}. First [Step
(i)] we find $\phi^1_c(\lambda,N)$ small enough to have survival: by
Proposition \ref{monotonicity0} the process survives for each $\phi$
smaller than $\phi^1_c(\lambda,N)$. Then [Step (ii)] we prove that the
process dies out for all $\lambda, N$ by taking $\phi\geq1$ if it
starts from a finite initial configuration and by taking $\phi>1$ if it
starts from $\eta_0 \in\OmN$. Finally in Step (iii) we use Corollary
\ref{monotonicity1} to obtain the existence of a critical parameter
$\phi_c(\lambda,N)$.\vadjust{\goodbreak}

%
\begin{figure}

\includegraphics{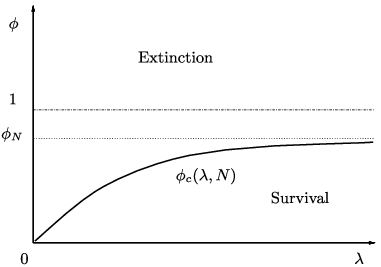}

\caption{Phase diagram of Model $I$ for a fixed $N>1$ and finite
initial configuration: by Theorem~\protect\ref{phasetransition1} there exists
a critical curve $\phi_c(\lambda,N)$ which converges to $0$ as $\lambda
$ goes to zero; it is always smaller than $1$. We conjecture that $\phi
_c(\lambda,N)$ is monotone and as $\lambda$ goes to infinity converges
to a value $\phi_N$ depending on the capacity $N$ of the model which is
strictly smaller than $1$ for each $N<\infty$.}
\label{phase1}
\end{figure}

Figure \ref{phase1} sketches the phase diagram in the $(\lambda,\phi)$ plane.
The model admits a phase transition with respect to the death rate $\phi
$ for each $N\geq2$, while the same process without migrations dies
out almost surely. The effect of a~migration is to move an individual
from a site in state $N$, where there is no possibility to give birth,
to a site with less than $N$ individuals, where it may reproduce
itself. Therefore even if there is no monotonicity with respect to
$\lambda$ (cf. Remark \ref{remNlambda}), this suggests that an
increase of $\lambda$ is good for survival. Contact interactions and
migrations work in a similar way, but small differences are present.
From a mathematical point of view an increase of the migration rate
does not favor ergodicity.

\section{Model II: The Allee effect}
\label{ModelII}

We translate the Allee effect into mathematical terms for a
metapopulation model. As in Model I, we fix a capacity~$N$ for all
sites, but we assume the death rate larger than (or equal to) the birth
rate when the density is small. Namely, fix a positive integer
$N_{A}\leq N$ and positive real numbers $\phi$, $\lambda$ and $\phi
_{A}\geq1$; the transitions are, for all $x\in S$, $y \in S$, $x\sim
y$, referring to the notation in~(\ref{generator})
\begin{eqnarray}\label{rates2}
\eta_t(x)&\to&\eta_t(x)+1 \qquad\mbox{at rate } P^1_{\eta_t(x)}=\eta_t(x)\I
_{\{\eta_t(x)\leq N-1\}}, \nonumber\\
\eta_t(x)&\to&\eta_t(x)-1
\nonumber
\\[-8pt]\nonumber
\\[-8pt]\\[-18pt]
\eqntext{\displaystyle\mbox{at rate } P^{-1}_{\eta_t(x)}=\eta
_t(x)\bigl(\phi_{A}\I_{\{\eta_t(x) \leq N_{A}\}}+\phi\I_{\{N_{A}<\eta_t(x)\}}\bigr),}\\
(\eta_t(x),\eta_t(y))&\to&\bigl(\eta_t(x)-1,\eta_t(y)+1\bigr)\nonumber\\
\eqntext{\displaystyle\mbox{at rate
}\frac{1}{2d}\Gamma^1_{\eta_t(x),\eta_t(y)}= \frac{\lambda}{2d}\I_{\{\eta_t(x)=N,\eta_t(y)<N\}}.}
\end{eqnarray}
We assume $\phi_A\geq1$ and $\phi_{A} \geq\phi$; in other words if
$\eta_t(x)\leq N_{A}$, then the death rate $\phi_{A}\eta_t(x)$ is
larger than (or equal to) the birth rate $\eta_t(x)$ because of the
Allee effect.
If $\eta_t(x)> N_{A}$, the most interesting situation is given by a
death rate $\phi\eta_t(x)$ smaller than or equal to the birth rate $\eta
_t(x)$, that is, $\phi\leq1$. If either $\phi\geq1$ and $\eta_0$ is
finite or $\phi>1$ and $\eta_0 \in\OmN$ the species gets extinct as
proved in Theorem \ref{phasetransition1}. If $N_{A}=0$ (no Allee
effect) or $N_{A}=N$ (death rate always larger than birth rate), there
is only one death rate, and we are back to Model I.

Since only births, deaths and migrations of at most one particle are
allowed, and the migration rate from $\eta_t(x)$ to $\eta_t(y)$ is
nondecreasing in~$\eta_t(x)$ and nonincreasing in $\eta_t(y)$,
attractiveness conditions are satisfied. One proves in a similar way
that Proposition \ref{monotonicity0} still holds for Model II either
with respect to $\phi_{A}$ or $\phi$, namely:
\begin{proposizione}
Let $(\xi_{t})_{t\geq0}$ and $(\eta_{t})_{t\geq0}$ be two Model
II-type processes with respective parameters $(\phi_{1}, \phi_{A,1},
\lambda,N, N_A)$ and $(\phi_{2}, \phi_{A,2}, \lambda,N, N_A)$ such that
$\phi_{1} \leq\phi_{2}$ and $\phi_{A,1} \leq\phi_{A,2}$. Then $(\xi
_{t})_{t\geq0}$ is stochastically larger than $(\eta_{t})_{t\geq0}$,
and $(\eta_{t})_{t\geq0}$ is attractive.
\label{monotonicity2}
\end{proposizione}

Corresponding Corollary \ref{monotonicity1} holds in a similar way.

We prove that the Allee effect changes the behavior of the system: for
any possible capacity $N$ and migration rates there exists an Allee
effect large enough for the species to become extinct.
\begin{teo}\label{phasetransition2}
Assume $\phi_A\geq1$, and let $\phi_c(\lambda,N)$ be the critical
parameter introduced in Theorem \ref{phasetransition1}. Then for all
$\lambda>0$, $0< N<\infty$, $0< N_{A} \leq N$:
\begin{longlist}[(ii)]
\item[(i)] if $\phi<\phi_c(\lambda,N)$, there exists a value $\phi
^{A}_{c}(\phi,\lambda,N,N_{A})$ such that if $\phi_{A}>\phi^{A}_{c}(\phi
,\lambda,N,N_{A})$, the species becomes extinct for any initial
configuration $\eta_0 \in\OmNless$, and if $\phi_{A}<\phi^{A}_{c}(\phi
,\lambda,N,N_{A})$ the species has a positive probability of survival;
\item[(ii)] if $\phi_c(\lambda,N)<\phi$ $(\leq\phi_A)$, the species
becomes extinct for any initial configuration $\eta_0 \in\OmNless$.
\end{longlist}
\end{teo}

This corresponds to the biological idea that random fluctuations, which
are present on each local population, plus the Allee effect doom even a
very large population.

The phase diagram of Model II depends on $\phi_A$. Proposition \ref
{phasetransition2} is not enough to construct a detailed phase diagram,
but it gives some information in this direction. Since for any $\phi$
and $\lambda$ there exists $\phi_A$ large enough for the species to
become extinct, one can choose $\phi_A$ large enough to reduce the
survival region in the $(\lambda,\phi)$ plane of Figure \ref{phase1}
for such fixed $\phi_A$.

In order to model the Allee effect, we require $\phi_A\geq1$ and $\phi
\leq1$. Note that from a biological point of view we just need $\phi
_A>\phi$, but if either $\phi_A>\phi> 1$ or $1>\phi_A>\phi$, by
monotonicity arguments we can work as in Model I.\vadjust{\goodbreak}

From a mathematical point of view, it would be interesting to
investigate a model where $\phi$ and $\phi_A$ play symmetric roles,
that is, $\phi_A\leq1$ and $\phi\geq1$. For fixed~$N$, $N_A$ and
$\lambda$ we prove that there is no $\phi_A$ such that there is
survival for all $\phi$ and no $\phi$, such that there is extinction
for all $\phi_A$.
\begin{teo}\label{phasetransitionnew}
For all $1<N_A<N$, $\lambda>0$:
\begin{longlist}[(ii)]
\item[(i)] for each $\phi>1$ there exists a value $\phi^{A}_{c}(\lambda
,N_{A},N,\phi)$ such that, if
$\phi_{A}<\phi^{A}_{c}(\lambda
,N_A,N,\phi)$, the process survives for any initial configuration $\eta
_0$ such that $|\eta_0|\geq1$ with positive probability;
\item[(ii)] for each $\phi_A<1$ there exists a value $\phi_{c}(\lambda
,N_{A},N,\phi_A)$ such that, if $\phi>\phi_{c}(\lambda,N_A,N,\phi_A)$,
the process dies out for any initial configuration $\eta_0 \in\OmNless$.
\end{longlist}
\end{teo}

\section{Model III: Mass migration as Allee effect solution}
\label{ModelIII}

We have already observed in Model I that a migration of a single
individual is good in absence of the Allee effect. The model without
migrations dies out, but if we add a~possible migration of one
individual there is a positive probability of survival. In Model II,
anyhow, a single individual migration may not be enough: even in the
supercritical region of $\phi$ in Model I there exists an Allee effect
strong enough for the species to become extinct.

Which strategy may a species adopt to reduce the Allee effect?

We show that, at least in theory, \textit{migrations of large flocks of
individuals improve the probability of survival for any Allee effect}.
A migration of many individuals in a new environment improves the
probability of a successful colonization avoiding a small density in
that new environment which is influenced by the Allee effect.

We introduce positive parameters $\phi_{A}$, $\phi$, $N_{A}$, $N$ such
that $0\leq N_{A}\leq N$, $\phi_A>1$, $\phi>0$ and we take birth and
death transitions as in Model II, but more general migration rates:
given $M \in\mathbb{N}$, $0 < M \leq N$, $y \sim x$ the transitions are
\begin{eqnarray}\label{rates3}
\eta_t(x)&\to&\eta_t(x)+1 \qquad\mbox{at rate } P^1_{\eta_t(x)}=\eta_t(x)\I
_{\{\eta_t(x)\leq N-1\}}, \nonumber\\
\eta_t(x)&\to&\eta_t(x)-1
\nonumber
\\[-8pt]\nonumber
\\[-8pt]\\[-18pt]
\eqntext{\displaystyle\mbox{at rate } P^{-1}_{\eta_t(x)}=\eta
_t(x)\bigl(\phi_{A}\I_{\{\eta_t(x) \leq
N_{A}\}}+\phi\I_{\{N_{A}<\eta_t(x)\}}\bigr),}
\\
(\eta_t(x),\eta_t(y))&\to& \bigl(\eta_t(x)-k,\eta_t(y)+k\bigr) \nonumber\\
\eqntext{\displaystyle\mbox{at rate }
\frac{1}{2d}\Gamma^k_{\eta_t(x),\eta_t(y)} =\frac{\lambda}{2d} \I_{\{\eta_t(x)-k\geq N-M, \eta
_t(y)+k\leq N\}}}
\end{eqnarray}
for $1\leq k\leq M$. In other words if $k \in\{1,2,\ldots,M\}$
individuals try to migrate from~$x$ to $y$, but if $\eta_t(y)+k>N$, the
migration does not happen. Notice that if $\eta_t(x)<N-M$ the migration
rate is null: \textit{individuals try to migrate only when there are more
than $N-M$ individuals on a site}. From a biological point of
view,\vadjust{\goodbreak}
this means that when there are few individuals, resources are enough
for all and there are no reasons to migrate. When $\eta_t(x) \geq N-M$
there is a positive probability of migration and the number of
individuals that may migrate is increasing with the population size. If
$\eta_t(x)=N-M+1$ we allow a migration of at most $1$ individual from
$x$ to a nearest neighbor site; when $\eta_t(x)=N-M+2$ we allow a
migration of either $1$ or $2$ individuals with rate $\lambda$ and so
on. If $\eta_t(x)=N=(N-M)+M$ we allow a migration of $1,2,\ldots$ to
the largest flock of $M$ individuals, where each migration occurs with
rate $\lambda$.

First of all we notice monotonicity properties.\vspace*{-2pt}
\begin{proposizione}
Let $(\xi_{t})_{t\geq0}$ and $(\eta_{t})_{t \geq0}$ be two Model
III-type processes with respective parameters $(\phi_{1}, \phi_{A,1},
\lambda,N, N_A)$ and $(\phi_{2}, \phi_{A,2}, \lambda,N, N_A)$ such
that $\phi_{1} \leq\phi_{2}$ and $\phi_{A,1} \leq\phi_{A,2}$. Then
$(\xi_{t})_{t\geq0}$ is stochastically larger than~$(\eta_{t})_{t\geq
0}$, and $(\eta_{t})_{t\geq0}$ is attractive.\vspace*{-2pt}
\label{monotonicity3}
\end{proposizione}

Corresponding Corollary \ref{monotonicity1} holds in a similar way.

In Model II we showed that a strong Allee effect dooms even a very
large population with a large migration rate. The strategy that the
species may adopt to reduce the Allee effect is to increase the number
of individuals which migrate: we prove that we can take a population
size $N$ and a maximal migration flock size $M$ large enough for the
species to survive \textit{for any Allee effect}.\vspace*{-2pt}

\begin{teo}
Let $d\geq2$. For all $\lambda>0$, $N_{A}\geq0$:
\begin{longlist}[(ii)]
\item[(i)] if $\phi< 1$ there exists $N_{c}(\phi,\lambda,N_{A})$ such
that for each $N>N_{c}(\phi,\lambda,N_{A})$, there exists $M(N_{A})$ so
that the process starting from $\eta_0$ with $|\eta_0|\geq1$ has
a~positive probability of survival for each $\phi_A<\infty$. Moreover if
$\eta_0 \in\OmN$ the process converges to a nontrivial invariant
measure for each $\phi_{A}<\infty$;
\item[(ii)] if $\phi\geq1$, the process becomes extinct for all $N$,
$\lambda$, $\phi_{A}>1$, $M$ and for any finite initial configuration.
If $\eta_0 \in\OmNless$ is not finite the process becomes extinct if
$\phi>1$.\vspace*{-2pt}
\end{longlist}
\label{phasetransition3}
\end{teo}

\begin{remark}
The proof of (i) (see Section \ref{proofs}) states that in order to
have survival we can take $M(N_{A})=N_{A}+1$. If $N_{A}=0$, this gives
$M(N_{A})=N_{A}+ 1=1$; only a migration of one individual is possible
and the process reduces to a Model I-type process: therefore Theorem
\ref{phasetransition1b} is a particular case of Theorem \ref{phasetransition3}.\vspace*{-2pt}
\end{remark}

Notice that $N_c(\phi,\lambda,N_A)$ does not depend on $\phi_A$. This
means that even if the Allee effect is the strongest one, if the
species lives and migrates in flocks large enough, survival is
possible.

Since there are many parameters the phase diagram is not easy to
construct; nevertheless Proposition \ref{phasetransition3} suggests
that one can choose $N$ and $M$ large enough to extend the survival
region in the $(\lambda,\phi)$-plane for fixed $\phi_A$,~$N$
and~$M$.\vadjust{\goodbreak}

\section{Model IV: Ecological equilibrium}
\label{ModelIV}

Real natural environments do not have any a priori bound on the
population size, but there is a kind of \textit{self-regulating
mechanism} that does not allow an ``explosion'' of the number of
individuals per site. \textit{Ecological equilibrium} has been introduced
in \cite{cfBertacchi} for restrained branching random walks (on a
connected, nonoriented graph $X$ with bounded geometry) with
transition rates
\begin{eqnarray*}
\eta(x)&\to&\eta(x)+1 \qquad \mbox{at rate } \sum_{y}\eta(y)p(y,x)c(\eta
(x)),\\
\eta(x)&\to&\eta(x)-1 \qquad \mbox{at rate } \eta(x),
\end{eqnarray*}
where $c\dvtx \mathbb{N}\to\mathbb{R}^+$ is a~nonincreasing function and
$P=(p(x,y))_{x,y \in X}$ is a~sto\-chastic matrix such that $p(x,y)>0$
only if $x\sim y$. The idea is that some restrictions on branching
random walks birth rates, given by the nonincreasing function $c(\cdot
)$ of the number of individuals, provide survival within nonexploding
populations. In particular, one interesting consequence of \cite{cfBertacchi}, Proposition 1.1,
is that one can find a function $c$
such that the process survives but $\limsup_{t\to\infty}\mathbb
{E}^{\eta_0}(\eta_t(x))<\infty$ uniformly for any bounded $\eta_0 \in
\Omega$ and $x\in X$.

We show that a similar mechanism leads to a similar conclusion on
different systems. Instead of taking births on neighboring sites as in
\cite{cfBertacchi}, we consider a nonincreasing birth rate in the
same local population, but we add migrations when the number of
individuals is larger than a fixed value~$N$. This means that the
restriction on birth rate does not change the migration rate: this is
not the case for the restrained branching random walk, where births in
a new site (which play the same role as migrations in Model~IV) depend
on the local population density.

We suppose that in our environment there is no maximal population size
as in previous models, and the birth rate is always positive. We also
assume that, when the population size is larger than $N$, the death
rate increases faster than the birth rate, hence the growth rate is
negative.

In order to simplify notation and proofs, we work on a modification of
Model I. Namely, given positive real values $\phi$, $\widetilde{\phi}$,
we take the following transitions, for each $x \in S$, $y \in S$,
$x\sim y$:
\begin{eqnarray}\label{rates4}
\eta_t(x)&\to&\eta_t(x)+1 \qquad\mbox{at rate } P^1_{\eta_t(x)}=\eta
_t(x),\nonumber\\
\eta_t(x)&\to&\eta_t(x)-1
\nonumber
\\[-8pt]\nonumber
\\[-8pt]\\[-18pt]
\eqntext{\displaystyle \mbox{at rate }P^{-1}_{\eta_t(x)}=\eta
_t(x)\bigl(\phi\I_{\{\eta_t(x)\leq N\}}+\Dphi\I_{\{N<\eta_t(x)\}}\bigr),}
\\
(\eta_t(x),\eta_t(y))&\to& \bigl(\eta_t(x)-1,\eta_t(y)+1\bigr) \nonumber\\
\eqntext{\displaystyle\mbox{at rate }
\frac{1}{2d}\Gamma^{1}_{\eta_t(x),\eta_t(y)}=\frac{\lambda}{2d}\I_{\{\eta_t(x)\geq N, \eta_t(y)<N\}}.}
\end{eqnarray}
This means that when the population size $\eta_t(x)$ is larger than
$N$, and the death rate $\widetilde{\phi}\eta_t(x)$ is larger than the
birth rate $\eta_t(x)$. A migration is allowed from a site with more
than $N$ individuals to a site with less than~$N$ individuals. Since we
are working without any a priori bound, we refer to construction
techniques in noncompact cases, and we restrict the state space to
$\widetilde{\Omega}\subseteq\Omega$ (see \cite{cfChenbook}, Chapter~13), where
\[
\widetilde{\Omega}:=\biggl\{\eta\in\Omega\dvtx  \sum_{x \in\mathbb{Z}^d}\eta
(x)\alpha(x)<\infty\biggr\},
\]
and $(\alpha(x))_{x \in\mathbb{Z}^d}$ is a positive sequence such that
$\sum_{x\in\mathbb{Z}^d}\alpha(x)<\infty$. Sufficient conditions for
existence and uniqueness of the process given in \cite{cfChenbook}, Chapter~13, are satisfied:
\begin{lemma}
There exists a unique Markov process with state space $\widetilde{\Omega
}$, generator (\ref{generator}) and rates (\ref{rates4}).
\label{existence}
\end{lemma}

Since births, deaths and migrations involve only one particle and the
migration rate is nondecreasing in $\eta_t(x)$ and nonincreasing in
$\eta_t(y)$ the process is attractive as in Model I, and a monotonicity
property (see Proposition~\ref{monotonicity1})\vspace*{1pt} holds in $\phi$ and
in~$\widetilde{\phi}$ for each initial configuration $\eta_0\in\widetilde
{\Omega}$. We prove that in some cases the process survives but does
not explode; that is, it does not die out, and the expected value on
each site is finite.
\begin{teo}
Let $\eta_0\in\Omega_n$ for some $n \in\mathbb{N}$ (so that $\eta_0
\in\widetilde{\Omega}$). For all $\lambda>0$, $\Dphi>1$:
\begin{longlist}[(ii)]
\item[(i)] for each $1<N<\infty$ there exists a critical value $\phi
_{c}(\lambda,N,\Dphi)>0$ such that if $\phi< \phi_{c}(\lambda,N,\Dphi
)$, the process has a positive probability of survival, and if $\phi>
\phi_{c}(\lambda,N,\Dphi)$ the process dies out;
\item[(ii)] for each $\phi<1$ there exists a value $N_{c}(\lambda,\phi
,\Dphi)>0$, such that if $N > N_{c}(\lambda,\phi,\Dphi)$, the process
has a positive probability of survival.

If the process survives, there exists $C_n<\infty$ so that
$ \lim_{t \to\infty}\mathbb{E}(\eta_{t}(x))\leq C_n$ for
each $x\in\mathbb{Z}^{d}$.
\end{longlist}
\label{phasetransition4}
\end{teo}

Note that the constant $C_n$ depends on the initial configuration.
Since the migration rate does not depend on the local population
density, we are not able to find such a constant $C$ independent of the
initial configuration, which was the case for the model treated in \cite
{cfBertacchi}.
\begin{remark}
In a similar way one can consider a Model III-type process without any
a priori bound by adding a death rate $\widetilde{\phi}\eta_t(x)$ when
the number of\vadjust{\goodbreak} individuals in a local population is larger than $N$. By
comparison arguments, even if a strong Allee effect is present, a mass
migration of large flocks of individuals leads to the survival of the
species, but the local populations do not explode.
\end{remark}

\section{Proofs}
\label{proofs}

We first recall a classical result involving random walks on a~finite
interval. Let $r_1, r_2 \in\mathbb{N}$ and $(X_{t})_{t\geq0}$ be a
discrete time random walk on $\{r_1,r_1+1, \ldots,r_2=r_1+n\}$ such that
\begin{eqnarray*}
i&\to& i+1 \mbox{ with probability }p, \qquad  i \in\{r_1,\ldots,r_2-1\}
,\\
i&\to& i-1 \mbox{ with probability }q, \qquad  i \in\{r_1+1,\ldots,r_2\}.
\end{eqnarray*}
We interpret this random walk as a game which ends when $X_t$ reaches
either~$r_1$ or~$r_2$, that we call respectively the ruin of the first
and the second players.

\begin{lemma}[{(Ruin Problem Formula, \cite{bookschinazi}, (4.4), Section~I.4)}]
Let $P_{r_2}(j)$ [resp., $P_{r_1}(j)$] be the
probability that the random walk starting at $j \in\{r_1+1, \ldots,
r_2-1\}$ reaches state $r_2$ before state $r_1$ (resp., state $r_1$
before $r_2$). Then
\[
1-P_{r_1}(j)=P_{r_2}(j)=\frac{1-(q/p)^{j-r_1}}{1-(q/p)^{n}}.
\]
\label{ruinlemma}
\end{lemma}

\subsection{Model I}
\subsubsection{\texorpdfstring{Proof of Proposition \protect\ref{monotonicity0}}{Proof of Proposition 3.1}}

We prove that if $\eta_{0} \leq\xi_{0}$, then $\eta_{t} \leq\xi_{t}$
for each $t>0$ a.s. This is an application of Theorem \ref{cns}; since
there is a~change of at most one particle per time, we check conditions
in Remark \ref{remM1}. The transition rates are given by (\ref
{rates1}), with $\phi=\phi_2$ for the process $(\eta_{t})_{t\geq0}$
and $\phi=\phi_1$ for $(\xi_{t})_{t\geq0}$. Conditions (\ref{C+1}) and
(\ref{C+10}) are the following: given $\eta\leq\xi$, if $\eta(y)=\xi
(y)$ and $\eta(x)\leq\xi(x)$
\begin{eqnarray*}
\eta(y)\I_{\{\eta(y)\leq
N-1\}}+\lambda\I_{\{\eta(x)=N,\eta(y)<N\}}&\leq&
 \xi(y)\I_{\{\xi(y)\leq N-1\}}+\lambda\I_{\{\xi(x)=N,\xi(y)<N\}},\\
\eta(y)\I_{\{\eta(y)\leq N-1\}}&\leq& \xi(y)\I_{\{\xi(y)\leq N-1\}}.
\end{eqnarray*}
Since $\eta\leq\xi$ and $\eta(x)=N$ imply $\xi(x)=N$, and since $\I_{\{
\eta(x)=N, \eta(y)<N\}}\leq\I_{\{\xi(x)=N, \xi(y)<N\}}$ if $\eta(y)=\xi
(y)$, the conditions are satisfied.

Conditions (\ref{C-1}) and (\ref{C-10}) are the following: if $\eta
(x)=\xi(x)$ and $\eta(y)\leq\xi(y)$,
\begin{eqnarray*}
&&\phi_{2}\eta(x)\I_{\{\eta(x)\leq N-1\}}+\lambda\I_{\{\eta(x)=N,\eta
(y)<N\}}\\
&&\qquad\geq \phi_{1}\xi(x)\I_{\{\xi(x)\leq N-1\}}+\lambda\I_{\{\xi
(x)=N,\xi(y)<N\}},\\
&&\phi_{2}\eta(x)\I_{\{\eta(x)\leq N-1\}}\geq\phi_{1}\xi(x)\I_{\{\xi
(x)\leq N-1\}},
\end{eqnarray*}
which hold since $\phi_2\geq\phi_1$ and $\I_{\{\eta(x)=N,\eta(y)<N\}
}\geq\I_{\{\xi(x)=N, \xi(y)<N\}}$, because $\eta(x)=\xi(x)$.

\subsubsection{\texorpdfstring{Proof of Theorem \protect\ref{phasetransition1}}{Proof of Theorem 3.2}}
\label{firstproof}

We prove it in three steps. In Step (i) we find $\phi^1_c(\lambda,N)$
small enough to have survival; in Step (ii)
we prove that the process dies out for all $\lambda, N$ by taking $\phi
\geq1$ if it starts from a finite initial configuration and by taking
$\phi>1$ it it starts from $\eta_0 \in\OmN$, and in Step (iii) we
get the existence of a critical parameter by monotonicity.

(i) We follow the idea in \cite{cfschivirus} by using the
comparison technique with oriented percolation (introduced in \cite
{cfBramsonDurrett}) explained in \cite{cfdurrettten}. Here and in
the subsequent proofs we think of the process as being generated by the
graphical representation; see \cite{cfdurrettten} for such a
construction. Suppose $d=2$. The proof in higher dimension is similar,
but the notation is more complicated. Denote~by
\begin{equation}
\qquad\cases{
e_{1}=(1,0), &\quad $\mathcal{N}=\{(m,n)\in\mathbb{Z}^2\dvtx  m+n \mbox{ is
even}\},$\vspace*{1pt}\cr
B=(-4L,4L)^{2}\times[0,T], & \quad
$B_{m,n}=(2mLe_{1},nT)+B,$\vspace*{1pt}\cr
I=[-L,L]^{2},& \quad$I_{m}=2mLe_{1}+I,$}
\label{blocks1}
\end{equation}
where $L$ and $T$ are integers to be chosen later. In other words
$B_{m,n}$ is the cube that we get by applying a translation of
$(2mLe_{1},nT)$ to $B$ and~$I_{m}$ the square we get by applying a
translation of $2mLe_{1}$ to $I$. Roughly speaking, the idea consists
of constructing boxes large enough so that with large probability the
species survives inside a box, and then to compare this evolution with
an oriented percolation model.

Let $(\eta_{t})_{t\geq0}$ be the process defined by generator (\ref
{generator}) with rates (\ref{rates1}). We consider a modification $\eta
^{m,n}_t$ of $\eta_t$: the process $(\eta^{m,n}_t)_{t\geq0}$ is
constructed through the graphical representation of $\eta_t$ in
$B_{m,n}$, but $\eta_t(x)=0$ for all~$x \notin B_{m,n}$ and $t\geq0$.
Let $m_{x,y}$ with $y \sim x$ be the Poisson process with rate $\lambda
/(2d)$ associated to a migration from $x$ to $y$. A migration from $x$
belongs to the graphical construction in $B_{m,n}$ if $x \in B_{m,n}$:
therefore an immigration to~$B_{m,n}$ from a site $y \notin B_{m,n}$
cannot happen for $\eta^{m,n}_t$, but we still consider the arrows of
emigrations from $B_{m,n}$. Their effect is the death of one individual
on the boundary of $B_{m,n}$. If $\eta_0(x)=\eta_0^{m,n}(x)=\I_{\{y\}
}(x)$ for some $y \in B_{m,n}$, $\eta_t\geq\eta^{m,n}_t$ by Remark \ref
{remM1} since if $x \notin B_{m,n}$, then $\eta^{m,n}_t(x)=0$;
otherwise conditions in Remark \ref{remM1} are satisfied for each pair
of sites $(x,y)$; see also Remark \ref{noninv}.

We say that $(m,n)$ is \textit{wet} if $\eta^{m,n}_t$ starting at time
$nT$ with at least one individual in $I_{m}$ is such that there is at
least one individual in $I_{m-1}$ and one individual in $I_{m+1}$ at
time $(n+1)T$. Otherwise the site is \textit{dry}. The event $G_{m,n}:=\{
(m,n)\mbox{ is wet}\}$ is measurable with respect to the graphical
construction in $B_{m,n}$: we prove that we can choose $L$ and $T$ such
that the probability of a site $(m,n)$ to be wet can be made
arbitrarily close to $1$ if $\phi$ is small enough.
By translation invariance it is enough to show it for $(0,0)$.
We call $\eta^{0,0}_{t}:=\xi_{t}$, we fix $L>0$ and we prove that for
each $\varepsilon>0$ there exists~$T$ and $\phi$ such that
\begin{equation}
\mathbb{P}\bigl((0,0) \mbox{ is wet}\bigr)\geq1-\varepsilon,
\label{genwet}
\end{equation}
that is, that if there exists one individual in a site $(i,j)\in
I_0=:I$, there is at least one individual both in $I_1$ and $I_{-1}$
with large probability.\vadjust{\goodbreak}

In order to prove it for $\phi$ small enough, we begin by showing that
it holds for a process with $\phi=0$ inside $B$: let $\widetilde{\mathbb
{P}}(\xi_t \in\cdot)$ denote the law of such a~process. This means
that each individual in box $B$ survives forever.

We choose a preferential path $(i,j), (i+1,j), \ldots,(L,j), (L+1,j)$:
we prove that there exists $T$ large enough so that the abscissas of
the rightmost and leftmost particles are respectively larger than $L$
and smaller than $-L$ with probability larger than $1-\varepsilon$, since
this is one possibility for the site $(0,0)$ to be wet.

A similar idea works for the leftmost particle. We conclude that if
$\phi=0$ for all $\varepsilon>0$, $\lambda>0$, $1<N<\infty$ there exists
$T=L\overline{T}$ such that
\begin{equation}
\widetilde{\mathbb{P}}\bigl((0,0) \mbox{ is wet}\bigr)>1-\varepsilon/2.
\label{pwet}
\end{equation}
Now we prove (\ref{genwet}) for $\phi$ small enough.
Let $A_{L}=A_L(\phi,N)$ be the time of the first death on the finite
box $(-4L,4L)^2$. If $A_L>T:=L\overline{T}$, for each $\varepsilon>0$ we
can take $\phi>0$ small enough for
\begin{eqnarray*}
\mathbb{P}\bigl((0,0) \mbox{ is wet}\bigr)&\geq&
\mathbb{P}\bigl((0,0) \mbox{ is wet}|A_{L}>T\bigr)\mathbb{P}(A_{L}>T)\\
&\geq&\widetilde{\mathbb{P}}\bigl((0,0) \mbox{ is wet}\bigr)e^{-\phi
N(8L)^2T}\geq 1-\varepsilon.
\end{eqnarray*}
Hence for all $\varepsilon>0$, $L>0$, $\lambda>0$, $1<N<\infty$ there
exists $T$ and $\phi^{1}_{c}(\lambda,N)>0$ such that if $\phi\leq\phi
^1_c(\lambda,N)$, then (\ref{genwet}) holds.

By comparing the process with an oriented percolation process, the
existence of an infinite path of wet sites corresponds to the existence
of individuals at all times, and for $\varepsilon$ small enough
percolation occurs; see \cite{cfdurrettten}. By monotonicity
(Proposition~\ref{monotonicity0}), the process survives for any $\phi
\leq\phi^{1}_{c}(\lambda,N)$.

(ii) Let $\xi_t$ be a continuous-time Galton--Watson
process without spatial structure starting from $|\eta_0|\leq\xi_0$
individuals. We couple the total number of particles of the two
processes. Each individual in both processes breeds at rate $1$ (except
for $\eta_t$ when the full carrying
capacity of the site is reached) and dies at rate $\phi$. Since we are
interested in the total number of particles, migrations do not count in
this coupling. Therefore $|\eta_t|\leq\xi_t$ for all $t\geq0$. If~$\xi
_0$ is finite and $\phi\geq1$, then the Galton--Watson process becomes
extinct; this implies that $\eta_t$ dies out for any $\phi\geq1$.

Assume now that $\eta_{0}\in\OmN$; we prove that the process becomes
extinct when $\phi>1$. By translation invariance, for each $t>0$,
\begin{eqnarray*}
\frac{d}{dt}\mathbb{E}(\eta_{t}(x))
&=&\mathbb{E}(\mathcal{L}\eta
_{t}(x))\\
&=&\mathbb{E}\biggl(\eta_{t}(x)\I_{\{\eta_{t}(x) \leq N-1\}
}-\phi\eta_{t}(x)+\sum_{y \sim x}\I_{\{\eta_{t}(y)=N,\eta_t(x)<N\}}\lambda/(2d)\\
&&\hspace*{113pt}\qquad{}-\sum_{y
\sim x}\I_{\{\eta_{t}(x)=N, \eta_t(y)<N\}}\lambda/(2d)\biggr)\\
&=&\mathbb{E}\bigl(\eta_{t}(x)\I_{\{\eta_{t}(x) \leq N-1\}}-\phi\eta
_{t}(x)\bigr)\leq(1-\phi)\mathbb{E}(\eta_{t}(x)),
\end{eqnarray*}
and by Gronwall's lemma the process converges to $0$ uniformly with
respect to $x$. By Corollary \ref{monotonicity1} the process dies out
for each $\phi>1$.

(iii) The claim follows by Steps (i), (ii) and
Corollary \ref{monotonicity1}. Starting from $\eta_0 \in\OmN$, the
existence of the upper invariant measure follows from attractiveness,
and it is nontrivial by Step (i).

\subsection{Model II}
\subsubsection{\texorpdfstring{Proof of Theorem \protect\ref{phasetransition2}}{Proof of Theorem 4.1}}

(ii) Since $\phi_A\geq\phi$, Model $I$ is stochastically larger than
Model $\mathit{II}$. If $\phi>\phi_c(\lambda,N)$, both of them die out by
Theorem \ref{phasetransition1}.

(i) Assume $\phi< \phi_c(\lambda,N)$ ($\leq1$ by Theorem
\ref{phasetransition1}). We follow the idea in~\cite{Berg}, Theorem~4.4,
and we compare the system with a subcritical percolation
process. We prove (i) when $d=2$ in order to simplify the notation
(the same proof works for all $d\geq1$). Let $(\eta_{t})_{t\geq0}$ be
a process with generator $(\ref{generator})$, rates $(\ref{rates2})$
and $\eta_{0}\in\OmN$. We define
\begin{equation}
\qquad\cases{
\mathcal{A}=[-2L,2L]^{2}\times[0,2T]; \qquad \mathcal
{B}=[-L,L]^{2}\times[T,2T],\vspace*{2pt}\cr
\mathcal{C}_{b}=\{(x,y,t)\in\mathcal{A}\dvtx t=0\},\vspace*{2pt}\cr
\mathcal{C}_{s}=\{(x,y,t)\in\mathcal{A}\dvtx |x|=2L \mbox{ or } |y|=2L
\},\vspace*{2pt}\cr
\mathcal{C}=\mathcal{C}_{b} \cup\mathcal{C}_{s}=\{(x,y,t)\in\mathcal
{A}\dvtx |x|=2L \mbox{ or } |y|=2L \mbox{ or } t=0\},}
\label{blocks2}
\end{equation}
where $T$ is a time to be fixed later.

In other words $\mathcal{C}$ is part of the boundary of the space--time
region $\mathcal{A}$, which contains the smaller region $\mathcal{B}$.
We construct a percolation process on $\mathcal{N}=\mathbb{Z}^{2}\times
\mathbb{Z}_{+}$ starting from $(\eta_{t})_{t\geq0}$. We consider for
each $(m,n,k)\in\mathcal{N}$ a~modification $\eta^{m,n,k}_t$ of~$\eta
_t$: the process $(\eta^{m,n,k}_t)_{t\geq0}$ is constructed through
the graphical representation of $\eta_t$ in $\mathcal{A}+(mL,nL,kT)$,
but $\eta^{m,n,k}_t(x)=N$ for all $x \in(mL,nL)+(-2L,2L)^2$, $t\leq
kT$ and $x \notin(mL,nL)+(-2L,2L)^2$ for all $t\geq0$. Therefore an
emigration from $\mathcal{A}+(mL,nL,kT)$ cannot happen and an
immigration from a site $y$ on the boundary of $(mL,nL)+[-2L,2L]^2$
after~$kT$ is always possible with rate $\lambda$. By Remarks \ref
{remM1} and \ref{noninv}, $\eta_t\leq\eta_t^{m,n,k}$ for all~$m$,
$n$, $k$ and $t\geq0$, since if $x \notin(mL,nL)+(-2L,2L)^2$, then
$\eta^{m,n}_t(x)=N$, otherwise conditions in Remark \ref{remM1} are
satisfied for each pair of sites $(x,y)$.

We say that a site $(m,n,k) \in\mathcal{N}$ is wet if there are no
individuals for the process $\eta^{m,n,k}_t$ in $\mathcal
{B}+(mL,nL,kT)$. A site is dry if it is not wet.

We show, through a series of lemmas, that the probability of a site to
be wet is as large as we want by taking $\phi_A$ large. By translation
invariance we prove it for $(0,0,0)$, and we denote $\eta^{0,0,0}_t:=\xi
_t$. Let $0<\phi_A<\infty$. First of all we prove that there exists a
time $S$ at which with large probability there is at most $1$
individual per site on $(-2L,2L)^2$ (Lemma \ref{lemB1}). After~$S$,
there exists a time $T$ such that there are no individuals in
$(-2L,2L)^2$ with large probability (Lemma \ref{lemB2}). Therefore
with large probability the only possibility of having one individual in
$\mathcal{B}$ is that an emigration from the boundary after time $T$
reaches $[-L,L]^2$ before $2T$: the last step consists in proving that
such an event has small probability.

We first introduce an auxiliary process whose transitions are not
translation invariant:
\begin{lemma}
Let $(\overline{\xi}_{t})_{t \geq0}$ be a process with only birth and
death rates: if $x \in(-2L,2L)^2$
\begin{equation}
\overline{P}^1_l(x)=\I_{\{l \leq N-1\}}(l+\lambda); \qquad \overline
{P}^{-1}_l(x)=l\bigl(\phi_A\I_{\{l\leq N_{A}\}}+\phi\I_{\{N_{A}<
l\}}\bigr)\hspace*{-30pt}
\label{procsuprates}
\end{equation}
and $\overline{\xi}_{t}(x)=N$ for all $x \notin(-2L,2L)^2$, $t\geq0$.
Then $(\overline{\xi}_{t})_{t\geq0}$ is stochastically larger than
$(\xi_{t})_{t \geq0}$.
\label{procsup}
\end{lemma}
\begin{pf} Both $\xi_{t}$ and $\overline{\xi}_t$ are equal to $N$
for each $t\geq0$ outside $(-2L,2L)^2$. By Remark \ref{noninv}, we
check the conditions in Remark \ref{remM1} for each pair of sites
$(x,y)$ with either $x$ or $y$ in $(-2L,2L)^2$. If $x \in(-2L,2L)^2$,
$(\overline{\xi}_{t}(x))_{t\geq0}$ is a~birth and death process whose
birth rate is the original one plus the largest immigration rate on $\xi
_{t}(x)$, and whose death rate is the original one plus the smallest
emigration rate on $\xi_{t}(x)$, which is null. For each $\eta\in
\Omega$,
\begin{eqnarray*}
P^1_{\eta(y)}+\Gamma^1_{\eta(x),\eta(y)}&\leq& \I_{\{\eta(y) \leq N-1\}
}\bigl(\eta(y)+\lambda\bigr)=\overline{P}^1_{\eta(y)},\\
P^{-1}_{\eta(x)}+\Gamma^1_{\eta(x),\eta(y)}& \geq& \I_{\{\eta(x)\leq
N_{A}\}}\phi_{A}\eta(x)+\I_{\{N_{A}< \eta(x)\}}\phi\eta(x)=\overline
{P}^{-1}_{\eta(x)};
\end{eqnarray*}
then all conditions are satisfied.
\end{pf}
\begin{lemma}
For all $\varepsilon>0$, $L$ there exists $S>0$ and $\phi_A$ such that
%
\begin{equation}
\mathbb{P}(G_{L}(S))>1-\varepsilon/6,
\label{numero}
\end{equation}
where $G_{L}(S)=\{\xi_{S}(x)\leq1 \mbox{ for each }x\in(-2L,2L)^{2}\}$.
\label{lemB1}
\end{lemma}
\begin{pf}
We prove (\ref{numero}) for $(\overline{\xi}_t)_{t\geq0}$ with law
$\overline{\mathbb{P}}(\overline{\xi}_t\in\cdot)$. By monotonicity
(Lemma \ref{procsup}) it will be true for $(\xi_t)_{t\geq0}$. For all
$\varepsilon>0$ and $L$ we take $S$ large enough so that the number of
visits $H^{S}_{x}$ to $0$ of $\overline{\xi}_{t}(x)$ before $S$ satisfies
%
\begin{equation}
\overline{\mathbb{P}}(H^{S}_{x}=0)\leq\frac{\varepsilon}{18(4L)^{2}}.
\label{tmp11}
\end{equation}
If there is at least one visit, we consider
\begin{eqnarray}\label{num}
&&\sum_{k=1}^{K}\overline{\mathbb{P}}\bigl(\overline{\xi}_{S}(x)>
1\vert H^{S}_{x}=k\bigr)\overline{\mathbb{P}}(H^{S}_{x}=k)
\nonumber
\\[-8pt]
\\[-8pt]
\nonumber
&&\qquad{} +\sum_{k=K+1}^{\infty}\overline{\mathbb{P}}\bigl(\overline{\xi}_{S}(x)>
1\vert H^{S}_{x}=k\bigr)\overline{\mathbb{P}}(H^{S}_{x}=k).
\end{eqnarray}
By taking~$K$ large enough the second sum (in which there are more than~$K$
hits to~$0$) is as small as we want. There are at least two
individuals in a~site\vadjust{\goodbreak} after the $i$th visit to $0$ only if the
exponential clock $B_{i}\sim \operatorname{Exp}(1+\lambda)$ [birth rate if $\overline
{\xi}_{t}(x)=1$] rings before the one of $D_{i}\sim \operatorname{Exp}(\phi_{A})$
[death rate if $\overline{\xi}_{t}(x)=1$]. Therefore for all $\varepsilon
>0$, $L$ and $K$ we can take $\phi_{A}$ large enough for the first sum
in (\ref{num}) to be smaller than
\begin{equation}\qquad
\sum_{k=1}^{K}\overline{\mathbb{P}}(\exists i\in\{1,2,\ldots, k\}
\dvtx B_{i}<D_{i})\leq K^{2}\frac{1+\lambda}{1+\lambda+\phi_{A}}\leq\frac
{\varepsilon}{18(4L)^{2}}.
\label{tmp13}
\end{equation}
By (\ref{tmp11}) and (\ref{tmp13}) for all $\varepsilon>0$, $L$
there exists $S$ and $\phi_{A}$ large enough for
\begin{equation}
\mathbb{P}((G_{L}(S))^{c})\leq(4L)^{2}\sup_{x \in
(-2L,2L)^{2}}\overline{\mathbb{P}}\bigl(\overline{\xi}_{S}(x)> 1\bigr)\leq
\varepsilon/6,
\label{epsfin1}
\end{equation}
and the claim follows.
\end{pf}
\begin{lemma}\label{lemB2}
For all $L$, $\varepsilon>0$ there exists $\overline{S}$ and $\phi_A$ such that
\[
\mathbb{P}\bigl(\overline{G}_{L}(S+\overline{S})\bigr)\geq1-\varepsilon/3,
\]
where $\overline{G}_{L}(S+\overline{S})=\{\xi_{S+\overline{S}}(x)=0\mbox
{ for each }x \in(-2L,2L)^{2}\}$, and $S$ is given by Lemma \ref{lemB1}.
\end{lemma}
\begin{pf}
If $G_L(S)$ holds, we take $\overline{S}$ small so that there are
neither births nor immigrations from the boundary $\mathcal{C}_{s}$
between $S$ and $\overline{S}$ and $\phi_A$ large so that all
individuals in $(-2L,2L)^{2}$ die before $\overline{S}$ with large
probability.\vspace*{1pt} Namely, given $D\sim \operatorname{Exp}(\phi_{A})$ and $B\sim \operatorname{Exp}
((1+\lambda)(4L-1)^{2})$, for all $\varepsilon>0$, $L$\vspace*{1pt} there exists
$\overline{S}$ small and $\phi_A(\overline{S})$ large enough for
\begin{eqnarray*}
\mathbb{P}\bigl(\overline{G}_{L}(S+\overline{S})\vert G_{L}(S)\bigr)&\geq& \mathbb
{P}(D<\overline{S})^{(4L-1)^{2}}\mathbb{P}(\overline{S}<B) \\
&\geq&\bigl(1-\exp(-\phi_A \overline{S})\bigr)^{(4L-1)^{2}}\exp
\bigl(-(1+\lambda)(4L-1)^2\overline{S}\bigr)\\
&\geq& 1-\varepsilon/6.
\end{eqnarray*}
If $T=S+\overline{S}$, given by the two previous lemmas,
\begin{equation}
\mathbb{P}((\overline{G}_{L}(T))^c)\leq\varepsilon/6+\varepsilon/6=\varepsilon/3,
\label{dryT}
\end{equation}
and the claim follows.
\end{pf}

Therefore $\xi_{T}(x)=0$ for each $x \in(-2L,2L)^2$ with large probability.
Since $P^1_0=0$, the only way to get an individual in $[-L,L]^2$
between times $T=S+\overline{S}$, given by the two previous lemmas, and
$2T$ is that a migration from $y \in\mathcal{C}_{T}=\{y=(y_1,y_2)\in
\mathcal{A}\dvtx |y_1|=2L \mbox{ or } |y_2|=2L \}$ gives birth to a chain of
individuals which reaches $[-L,L]^2$ in a time smaller than $T$.
Suppose that $\xi_t(y)=N$ for all $y \in\mathcal{C}_T$ and $t \in
[T,2T]$. By monotonicity it will be true for any smaller configuration.
We fix $\widetilde{K}$ large so that the number of
emigra\-tions~$E_{T,\mathcal{C}_T}$ from $\mathcal{C}_{T}$ to $(-2L,2L)^2$ from time
$T$ to $2T$ is larger than $\widetilde{K}$ with probability smaller
than $\varepsilon/3$.

After one migration, with probability smaller than
$(1+\lambda)/(\phi_A+1+\lambda)$ there is a new birth or a new
immigration at $x$ before the death of the individual.
If the number of such migrations is smaller than $\widetilde{K}$, by
taking~$\phi_A$ large enough
\begin{equation}\qquad
\mathbb{P}\bigl((0,0,0) \mbox{ is dry}\mid\overline{G}_{L}(T),E_{T,\mathcal
{C}_{T}}\leq\widetilde{K}\bigr)
\leq\frac{\widetilde{K}(1+\lambda)}{\phi_A+1+\lambda}<\varepsilon/3.
\label{dry-temp}
\end{equation}
By (\ref{dryT}) and (\ref{dry-temp}) we get
\begin{eqnarray}\label{dry}
\mathbb{P}\bigl((0,0,0) \mbox{ is dry}\bigr)&< & \mathbb{P}\bigl((0,0,0)
\mbox{ is dry}\mid\overline{G}_{L}(T)\bigr)+\varepsilon/3 \nonumber\\
&= & \mathbb{P}\bigl((0,0,0) \mbox{ is dry}\mid\overline
{G}_{L}(T),E_{T,\mathcal{C}_{T}}>\widetilde{K}\bigr)\mathbb{P}(E_{T,\mathcal
{C}_{T}}>\widetilde{K})
\nonumber
\\[-8pt]
\\[-8pt]
\nonumber
&&{}+\mathbb{P}\bigl((0,0,0) \mbox{ is dry}\mid\overline
{G}_{L}(T),E_{T,\mathcal{C}_{T}}\leq\widetilde{K}\bigr)\mathbb
{P}(E_{T,\mathcal{C}_{T}}\leq\widetilde{K})\\
&&{}+\varepsilon/3 < \varepsilon/3+\varepsilon/3+\varepsilon/3=\varepsilon.\nonumber
\end{eqnarray}
Now we construct a dependent percolation model such that the
probability of a site to be wet is as large as we want. For all
$(m,n,k)$ and $(x,y,z)$ in $\mathcal{N}$ such that $k \leq z$ and the
intersection between $(mL,nL,kT)+\mathcal{A}$ and $(xL,yL,zT)+\mathcal
{A}$ is not empty we draw an oriented edge. Notice that the probability
of a site $(m,n,k)$ to be wet depends only on the existence of a path
of individuals within $(mL,nL,kT)+\mathcal{A}$; since each block
intersects only a finite number of other blocks, there exists $K$ such
that all sets of sites in $\mathcal{N}$ with distance larger than $K$
are independently wet. Here the distance is the minimal number of edges
(without orientation) connecting two sites. Therefore this is a
dependent percolation model with finite range of interactions.

By monotonicity, the probability of having an individual in
metapopulation model $(\eta_t)_{t\geq0}$ in $(mL,nL,kT)+\mathcal{B}$
is smaller than the probability of the existence of a path of dry sites
in the percolation model with endpoint $(m,n,k)$ starting from
$(y,z,0)$ for some $(y,z)\in\mathbb{Z}^2$. By working as in
\cite{Berg},
proof of Theorem $4.4$, for any given site $x\in S$ there
exists a random time~$T_{x}$ a.s. finite after which there will never
be any individual.
Let $A$ be a finite subset of $\mathbb{Z}^{d}$ and $T_{A}:=\max\{T_{x},
x \in A\}$. By monotonicity, $T_A$ may be chosen uniformly in the
initial configuration $\eta_0$. Given $\eta_0\in\OmN$, let $\overline
{\nu}$ be the invariant measure $ \lim_{t \to\infty}\delta
_{\eta_0}T(t)$ [where $T(t)$ is the semi-group of the process], which
exists by attractiveness. For each finite set $A \subset\mathbb{Z}^d$
\[
\overline{\nu}\bigl(\xi\in\Omega\dvtx  \xi(x)>0 \mbox{ for some }x\in A\bigr)=0.
\]
Since $\overline{\nu}$ gives null probability to each set of
configurations with at least one individual, it concentrates on the
empty configuration; that is, $\overline{\nu}\sim\delta_{\underline
{0}}$, and ergodicity follows.

\subsubsection{\texorpdfstring{Proof of Theorem \protect\ref{phasetransitionnew}}{Proof of Theorem 4.2}}
(i) We work as in proof of Theorem \ref{phasetransition1} with the
same notation:\vadjust{\goodbreak} we suppose $d=2$, we use (\ref{blocks1}) in order to
make a~comparison with an oriented percolation model and we define for
each $(m,n)$ a modification $\eta_t^{m,n}$ of the process in the same
way. A site $(m,n)\in\mathbb{Z}^2$ is wet if $\eta^{m,n}_t$ starting
at time $nT$ with at least one individual in $I_{m}$ is such that there
is at least one individual in $I_{m-1}$ and one individual in $I_{m+1}$
at time $(n+1)T$. By translation invariance we work on $\xi_t:=\eta
^{0,0}_t$. We will prove the analog of (\ref{genwet}).

We start with one individual at $x=(i,j)\in I$, and we choose a
preferential path $(i,j), (i+1,j), \ldots,(L,j), (L+1,j)$: if there
exists $T$ such that the abscissas of the rightmost and leftmost
particles of $\xi_t$ are respectively larger than $L$ and smaller than
$-L$ at $T$, then site $(0,0)$ is wet. We begin by working with $\phi
_A=0$ and call $\widetilde{\mathbb{P}}(\xi_t \in\cdot)$ the law of the
process in this case.

We fix $L>0$. We wait until in $(i,j)$ we have a stack of $N_A$
individuals: since $\widetilde{A}=\{N_A,N_A+1,\ldots, N\}$ is an
absorbing set (because $\phi_A=0$), after a finite time the local
population size reaches $N$ and migrates to $(i+1,j)$. Then we wait for
another migration from $(i+1,j)$ to $(i+2,j)$, and so on, so that in a
finite time we reach $(L+1,j)$. We work in the same way for the
leftmost particle. We conclude that if $\phi_A=0$ for all $\varepsilon
_1>0$, $\lambda>0$, $1<N<\infty$ there exists $T_{\varepsilon_1}$ such
that $\xi_T(x)\geq N_A$ for each $x \in[-2L,2L]^2$ with probability
larger than $1-\varepsilon_1$: hence
%
\begin{equation}
\widetilde{\mathbb{P}}\bigl((0,0) \mbox{ is wet}\bigr)\geq1-\varepsilon_1.
\label{tmm}
\end{equation}
Suppose $\phi_A>0$. For each $\varepsilon>0$ there exists $\varepsilon_1$ and
$T_{\varepsilon_1}$ large so that (\ref{tmm}) holds and $\phi_A$ small so
that the probability of a death before $T_{\varepsilon_1}$ is as small as
we want. Therefore
\[
\mathbb{P}\bigl((0,0) \mbox{ is wet}\bigr)\geq1-\varepsilon.
\]
We conclude that for all $L$, $\varepsilon>0$ and $(m,n)\in\mathcal{N}$
the event $G_{m,n}=\{(m,n)\allowbreak \mbox{ is wet}\}$, which is measurable with
respect to the graphical construction in~$B_{m,n}$, satisfies $\mathbb
{P}(G_{m,n})>1-\varepsilon$ by taking $T$ large and $\phi_{A}$ small. By
comparison arguments with oriented percolation we get the result.

(ii) The idea is that even for $\phi_A$ small, there exists
$\phi$ large so that the probability that the population size reaches
$N$ and then one individual migrates is small. One can prove the result
by repeating the steps we did to prove Theorem \ref
{phasetransition2}.

\subsection{Model III}
\subsubsection{\texorpdfstring{Proof of Proposition \protect\ref{monotonicity3}}{Proof of Proposition 5.1}}

We check the sufficient conditions for stochastic order from Theorem
\ref{cns}. We call $\{P^{\cdot}_{\cdot,\cdot},\Gamma^{\cdot}_{\cdot
,\cdot}\}$ the rates of $(\xi_t)_{t\geq0}$ and $\{\widetilde{P}^{\cdot
}_{\cdot,\cdot},\widetilde{\Gamma}^{\cdot}_{\cdot,\cdot}\}$ the ones of
$(\eta_t)_{t\geq0}$. They are given by
\begin{eqnarray*}
P_{\alpha}^{-1}&=&\cases{
\alpha\phi_{A,1}, & \quad $\mbox{if }\alpha\leq N_{A},$\vspace*{2pt}\cr
\alpha\phi_1, & \quad$\mbox{if } \alpha> N_{A},$}\qquad
\widetilde{P}_{\alpha}^{-1}=\cases{
\alpha\phi_{A,2}, &\quad $ \mbox{if }\alpha\leq N_{A},$\vspace*{2pt}\cr
\alpha\phi_2, & \quad $\mbox{if }\alpha> N_{A},$}
\\[-2pt]
P_{\beta}^{1}&=&\widetilde{P}_{\beta}^{1}=\beta\qquad \mbox{ if } \beta<N,\\[-2pt]
\Gamma_{\alpha, \beta}^{k}&=&\widetilde{\Gamma}_{\alpha, \beta}^{k}=
\lambda\qquad \mbox{if } \alpha-k\geq N-M \mbox{ and }\beta+k\leq N.
\end{eqnarray*}
Let $\alpha\leq\gamma$, $\beta\leq\delta$. We evaluate the terms in
condition (\ref{C+}). The birth rates give
\begin{eqnarray*}
\sum_{k\in X\dvtx k > \delta-\beta+j_{1}}\widetilde{P}_{\beta}^{k}&=&\I_{\{1
> \delta-\beta+j_{1}\}}\widetilde{P}_{\beta}^{1}=\beta\I_{\{\beta=\delta
<N, j_{1}=0\}},\\[-2pt]
\sum_{l\in X\dvtx l >j_{1}}P_{\delta}^{l}&=&\I_{\{1 >j_{1}\}}P_{\delta
}^{1}=\delta\I_{\{j_{1}=0,\delta<N\}},
\end{eqnarray*}
thus
\begin{equation}
\sum_{k \in X\dvtx k > \delta-\beta+j_{1}}\widetilde{P}_{\beta}^{k}\leq\sum
_{l \in X\dvtx l >j_{1}}P_{\delta}^{l}.
\label{attbirthIV}
\end{equation}
The death rates give
\begin{eqnarray*}
\sum_{l\in X\dvtx l > \gamma-\alpha+h_{1}}P_{\gamma}^{-l}&=&\I_{\{1 > \gamma
-\alpha+h_{1}\}}P_{\gamma}^{-1}
\\[-2pt]
&=&\gamma\I_{\{\gamma=\alpha, h_{1}=0\}}\bigl( \phi_{A,1}\I_{\{\gamma\leq
N_{A}\}}+\phi_1\I_{\{N_{A}<\gamma\}} \bigr),\\[-2pt]
\sum_{k\in X\dvtx k >h_{1}}\widetilde{P}_{\alpha}^{-k}&=&\I_{\{1 >h_{1}\}
}\widetilde{P}_{\alpha}^{-1}=\alpha\I_{\{h_{1}=0\}}\bigl( \phi_{A,2}\I
_{\{\alpha\leq N_{A}\}}+\phi_2\I_{\{N_{A}<\alpha\}} \bigr),
\end{eqnarray*}
thus
\begin{equation}
\sum_{k\in X\dvtx  k >h_{1}}\widetilde{P}_{\alpha}^{-k}\geq\sum_{l\in X \dvtx l
> \gamma-\alpha+h_{1}}P_{\gamma}^{-l}.
\label{attdeathIV}
\end{equation}
Now we consider the migration rates
\begin{eqnarray*}
\sum_{k \in I_{a}}\widetilde{\Gamma}_{\alpha,\beta}^{k}&=&\sum_{k \in
I_{a}}\lambda\I_{\{k\leq(\alpha-N+M)\wedge(N-\beta)\}},\\[-2pt]
\sum_{l \in I_{b}}\Gamma_{\gamma,\delta}^{l}&=&\sum_{l \in I_{b}}\lambda
\I_{\{l\leq(\gamma-N+M)\wedge(N-\delta)\}}.
\end{eqnarray*}
By (\ref{Ia})--(\ref{Id}), setting $l=k-\delta+\beta$,
%
\begin{eqnarray*}
&&\hspace*{-4pt}\sum_{k \in I_{a}}\lambda \I_{\{k\leq(\alpha-N+M)\wedge(N-\beta)\}}\\[-2pt]
&&\hspace*{-4pt}\qquad=  \lambda\Biggl|\bigcup_{i=1}^{K}\{m_{i}-\delta+\beta\geq l>j_{i}\}
\cap\{0\leq l\leq(\alpha-N+M-\delta+\beta)\wedge(N-\delta)\}\Biggr|\\[-2pt]
&&\hspace*{-4pt}\qquad\leq\lambda\Biggl|\bigcup_{i=1}^{K}\{\gamma-\alpha+m_{i}\geq l>j_{i}\}
\cap\{0\leq l\leq(\gamma-N+M)\wedge(N-\delta)\}\Biggr|\\[-2pt]
&&\hspace*{-4pt}\qquad=\sum_{l \in
I_{b}}\lambda\I_{\{l\leq(\gamma-N+M)\wedge(N-\delta)\}}
\end{eqnarray*}
since $\delta\geq\beta$ and $\gamma\geq\alpha$. Therefore
%
\begin{equation}
\sum_{k \in I_{a}}\widetilde{\Gamma}_{\alpha,\beta}^{k}\leq\sum_{l \in
I_{b}}\Gamma_{\gamma,\delta}^{l}.
\label{attjumpIV1}
\end{equation}
In a similar way we note that
\begin{eqnarray*}
\sum_{k \in I_{d}}\widetilde{\Gamma}_{\alpha,\beta}^{k}&=&\sum_{k \in
I_{d}}\lambda\I_{\{k\leq(\alpha-N+M)\wedge(N-\beta)\}}, \\[-2pt]
 \sum_{l\in I_{c}}\Gamma_{\gamma,\delta}^{l}&=&\sum_{l \in I_{c}}\lambda\I_{\{
l\leq(\gamma-N+M)\wedge(N-\delta)\}};
\end{eqnarray*}
then, by setting $k=l-\gamma+\alpha$, the sum $ \sum_{l
\in I_{c}}\lambda\I_{\{l\leq(\gamma-N+M)\wedge(N-\delta)\}}$ is equal to
\begin{eqnarray*}
&&\lambda\Biggl|\bigcup_{i=1}^{K}\{m_{i}-\gamma+\alpha\geq k>h_{i}\}\cap
\{0\leq k\leq(\alpha-N+M)\wedge(N-\delta-\gamma+\alpha)\}\Biggr|\\[-2pt]
&&\qquad\leq\lambda\Biggl|\bigcup_{i=1}^{K}\{\delta-\beta+m_{i}\geq k>h_{i}\}
\cap\{0\leq k\leq(\alpha-N+M)\wedge(N-\beta)\}\Biggr|\\[-2pt]
&&\qquad=\sum_{k \in I_{d}}\lambda\I_{\{k\leq(\alpha-N+M)\wedge(N-\beta)\}}
\end{eqnarray*}
since $N-\delta-\gamma+\alpha\leq N-\beta$. Hence
%
\begin{equation}
\sum_{k \in I_{d}}\widetilde{\Gamma}_{\alpha,\beta}^{k}\geq\sum_{l \in
I_{c}}\Gamma_{\gamma,\delta}^{l}.
\label{attjumpIV2}
\end{equation}
We get condition (\ref{C+}) by using (\ref{attbirthIV}) and (\ref
{attjumpIV1}) and condition (\ref{C-}) from~(\ref{attdeathIV}) and~(\ref{attjumpIV2}).

\subsubsection{\texorpdfstring{Proof of Theorem \protect\ref{phasetransition3}}{Proof of Theorem 5.1}}

We follow the idea in \cite{cfschiaggr}, proof of Theorem~2. We
assume $d=2$. If $d\geq2$ the proof works in a similar way. We
take~$N$, $M$ such that $N-M>N_A$. We fix $x \in\mathbb{Z}^d$, and we start
from an initial configuration $\eta_0(x)=N-M$ and $\eta_0(z)=0$ for
each $z \neq x$. We prove that starting from $\eta_0$, after a finite
time there is a migration of the largest flock of $M$ ($N_A <M <
N-N_A$) individuals into a site $y \sim x$ which will give birth to
$N-M$ individuals in the new site with large probability.

For each $x\in\mathbb{Z}^2$ we consider a modification $(\eta
^{x}_t)_{t\geq0}$ constructed through the graphical representation in
$I_x: =[x-1,x+1]^2$ such that $\eta_t^x(z)=0$ for each $z\notin I_x$ and
$t\geq0$: we take into account births, deaths and emigrations from
$x$, births and deaths on each $y\sim x$, but we replace migrations of
$k$ individuals from $y\sim x$ to $x$ by the death of $k$ individuals
on $y$. For $y \sim x$, let
\begin{eqnarray*}
&&E_{x,y}:=\{\mbox{There exists } T<\infty \mbox{ such that }\eta
^{x}_T(y)=N-M |\\
&&\hspace*{94pt}\eta^x_0(x)=N-M,\eta^x_0(z)=0, \forall z\sim x\}.
\end{eqnarray*}
Note that $\eta_t\geq\eta_t^x$ [it follows by construction from the
graphical representation, since $\eta^x$ is built from $\eta$;
alternatively one can check conditions (\ref{C+})--(\ref{C-}) by Remark
\ref{noninv}]. In particular before $T$ the process $\eta^x_t(x)$
behaves as~$\eta_t(x)$ without immigration, and $\eta^x_t(y)$ behaves
as $\eta_t(y)$. Therefore if $E_{x,y}$ occurs, $\eta_T(y)\geq  N-M$.

To make a comparison with an oriented percolation model, we follow \cite
{cfkuulasmaa}: between any two nearest neighbor sites $x$, $y$ in
$\mathbb{Z}^{2}$ we draw a directed edge from $x$ to~$y$, denoted by
$[x,y\rangle$: we say that one edge is \textit{open} if $E_{x,y}$ happens.
This defines a \textit{locally dependent random graph} since $E_{x,y}$
depends only on the graphical representation in $I_x$. The probability
of the directed edge $[x,y\rangle$ to be open is the same for all edges
$[x,y\rangle$ and $E_{x,y}$ and $E_{z,t}$ are independently open if
$x\neq z$.

We prove that for each $\varepsilon>0$ there exists $N$ large enough for
$\mathbb{P}(E_{x,y} \mbox{ is}\allowbreak \mbox{open})\geq1-\varepsilon$.

By translation invariance we suppose $x=0$. We prove that the following
events happen with large probability: first of all, starting from
$N-M$,\vspace*{1pt} the number of visits to $N-M+1$ of $\eta^0_t(x):=\xi_t$ before
visiting $N_A$ is at least~$N^3$ (Lemma~\ref{N-M}); if there are at
least~$N^3$ visits to $N-M+1$, there are at least~$N^2$ visits to $N$
(Lemma~\ref{enne}) before reaching $N_A$; if there are at least $N^2$
visits to $N$, there are at least $N^{1/2}$ mass migrations of $M>N_A$
individuals to a fixed site $y\sim0$ (Lemma~\ref{lemstep3}); finally
one of these mass migrations gives birth to $N-M$ individuals on $y$
before reaching $N_A$ with large probability.


(I) First of all we prove that the number of visits $R^{\xi
}_{N,M}$ to $N-M+1$ before reaching $N_{A}$ of the process $(\xi
_{t})_{t\geq0}$ starting at $N-M$ is large with large probability.
\begin{lemma}\label{N-M}
%
\begin{equation}
\lim_{N \to\infty}\mathbb{P}(R^{\xi}_{N,M}\geq N^{3})=1.
\label{stimaI}
\end{equation}
\end{lemma}

\begin{pf}
We construct a process $(\zeta_t)_{t\geq0}$ with state space $A:=\{
N_{A},N_{A}+1, \ldots, N-M+1\}$ by coupling with $\xi_t$ in the following
way:
\begin{itemize}
\item if $N_A\leq\xi_t\leq N-M+1$, then $\zeta_t=\xi_t$;

\item if $\xi_t\geq N-M$, then $\zeta_t=N-M+1$;
\end{itemize}
and $N_A$ is an absorbing state for $(\zeta_t)_{t\geq0}$. Each time
that $\zeta_t$ hits $N-M+1$ (an event which\vspace*{-1pt} can happen only from
below, i.e., if $\zeta_t$ moves from $N-M$ to $N-M+1$), so does $\xi
_t$. Therefore we count the number of visits $R^{\zeta}_{N,M}$ to
$N-M+1$ of the process $\zeta_t$ starting at $N-M+1$.
Note that $\xi_t$ comes back to state $N-M$ after visiting $N-M+1$ at
an a.s. finite time $T_M$ which satisfies
\begin{equation}
\mathbb{P}(T_M>t)\leq e^{-\lambda t}
\label{Txi}
\end{equation}
for each $N$, since if a mass migration of $\xi_t-(N-M)$ particles
occurs then $\xi_{t}$ comes back to $N-M$ with rate $\lambda$. The
skeleton of the process $(\zeta_t)_{t\geq0}$ moves as a discrete time
random walk on $A$ which comes back to $N-M$ after visiting $N-M+1$
with probability one, probability of birth $p=1/(1+\phi)$ and
probability of death $1-p$. We prove that
%
\begin{equation}
\lim_{N\to\infty}\mathbb{P}(R^{\zeta}_{N,M}\geq N^{3})=1.
\label{stimaRn}
\end{equation}
The probability that, starting at $N-M$, $\zeta_t$ returns to $N-M+1$
before visiting $N_A$ is given by Lemma \ref{ruinlemma} with $r_1=N_A$,
$r_2=N-M+1$, $j=N-M$, $q/p=\phi$. Since after visiting $N-M+1$ the walk
returns to $N-M$, by the Markov property [$P_{N-M+1}(N-M)$ is the
notation in Lemma \ref{ruinlemma}],
\begin{eqnarray*}
\mathbb{P}(R^{\zeta}_{N,M}\geq N^3)&=&\bigl(P_{N-M+1}(N-M)\bigr)^{N^3}\geq (1-\phi
^{N-M-N_A})^{N^3}\\
&\geq& \exp(-CN^3\phi^{N-M-N_A})
\end{eqnarray*}
so that (\ref{stimaRn}) [and then (\ref{stimaI})] follows since $\phi<1$.
\end{pf}

(II) 
Let $R_N^{\xi}$ be the number of visits of $(\xi_t)_{t\geq0}$ to $N$
before visiting $N_A$ starting at $N-M$.
\begin{lemma}\label{enne}
%
\begin{equation}\label{stimaII}
\lim_{N \to\infty}\mathbb{P}(R_{N}^\xi\geq N^{2})=1.
\end{equation}
\end{lemma}
\begin{pf} By (\ref{stimaI})
\begin{equation}
\mathbb{P}(R^\xi_{N}< N^{2})= \mathbb{P}(R^\xi_N< N^{2}|R^\xi
_{N,M}\geq N^{3})+o(1),
\label{tempRN}
\end{equation}
where $ \lim_{N\to\infty}o(1)=0$. We define a family of
i.i.d. random variables\break $\{X_i\}_{i=1,\ldots, N^3}$ such that $X_i=1$ if
$\xi_t$ reaches $N$ before $N-M$ at the $i$th visit to $N-M+1$, $0$
otherwise. One possibility for $X_i$ to be one is the birth of $N$
individuals without any death or mass migrations. Such an event has
probability larger than
\[
p_N:=\biggl(\frac{1}{1+\phi+\lambda M/N}\biggr)^M\geq\biggl(\frac{1}{1+\phi
+\lambda M}\biggr)^M=:p
\]
which does not depend on $N$. Therefore if $\overline{Y}$ is a binomial
random variable\vspace*{-1pt} with parameters $p$ and $N^3$, then $\mathbb{P}(\sum
^{N^3}_{i=1}X_i<N^{2})\leq\mathbb{P}(\overline{Y}<N^{2})$, which
converges to zero as $N$ goes to infinity by the central limit theorem.
\end{pf}

(III) Step (II) states that for each $\varepsilon>0$ we are able
to take $N$ large enough so that with probability larger than
$1-\varepsilon$ the process $\xi_t$ reaches $N$ at least $N^{2}$ times. We
prove that in this case, for a fixed $y\sim0$, with large probability
there is a migration $E_{N}=E_{N}(0,y)$ of $M$ individuals from $0$ to~$y$
at least $N^{1/2}$ times.
\begin{lemma}\label{lemstep3}
\[
\lim_{N\to\infty}\mathbb{P}(E_{N}\geq N^{1/2})=1.
\]
\end{lemma}

\begin{pf}
Notice that when $\xi_{t}$ visits $N$ there is a migration of $M$
individuals from $0$ onto site $y$ with rate $\lambda/(2d)$: if this is
not the case, either a death at $x$ or a different migration (i.e.,
less than $M$ individuals onto $y$ or a migration onto $z\sim x$, $z
\neq y$) occurs with rate smaller than $N\phi+\lambda M
(2d-1)/(2d)+(M-1)\lambda/(2d)$. Thus the probability of a migration to
$y$ of $M$ particles is larger than $\lambda/(2d(\lambda M+N\phi
))$.

The rest of the proof is identical to Step 2 of \cite{cfschiaggr}, proof of Theorem
$2$: the key point is that conditioning on $\{R^{\xi
}_{N} \geq N^{2}\}$, $E_{N}$ is larger than a binomial random variable
$V_{N}$ with parameters $N^{2}$ and $\lambda/(2d(\lambda M +N\phi))$,
such that $(V_N-\mathbb{E}(V_N))/(N^{1/2+a})$ converges to $0$ in
probability for all $a > 0$. The claim follows by taking $a \in(0,1/2)$.
\end{pf}
%

(IV) We show that given at least $N^{1/2}$ emigrations from
$0$ to $y$ of $M>N_A$ particles, at least one of these flocks of
individuals generates at least $N-M+1$ individuals on $y$ before
reaching size $N_A$. Every time there is a migration of $M$ individuals
to $y$, since $M=M(N_{A})>N_{A}$, the process $(\eta^0_t(y))_{t\geq0}$
is a birth and death chain with transitions
\begin{eqnarray*}
\eta^0_t(y)&\to&\eta^0_t(y) + 1 \qquad\mbox{at rate } \eta^0_t(y) \I_{\{
N_A<\eta^0_t(y)\leq N-M+1\}},\\
\eta^0_t(y) &\to&\eta^0_t(y)-1 \qquad\mbox{at rate } \eta^0_t(y)\phi\I_{\{
N_A <\eta^0_t(y) \leq N-M+1\}}.
\end{eqnarray*}
Take the same chain on $\{N_{A},\ldots,\infty\}$. Since $\phi<1$, the
chain is transient; therefore there is a positive probability $q(\phi)$
that starting at $M>N_{A}$ the chain will go on to infinity. The claim
follows as in Step 3 of \cite{cfschiaggr}, proof of Theorem~2,
since $N^{1/2}$ visits are enough for the probability to reach $N-M+1$
at least one time to approach $1$.

We conclude that for each $\varepsilon>0$ there exists $N$ and $T_{x,y}$
large such that~$E_{x,y}$ occurs in a finite time $T_{x,y}$ with
probability larger than $1-\varepsilon$.

(V) Finally we conclude the comparison with the oriented
percolation model on~$\mathbb{Z}^2$. We say that percolation occurs if
there exists an infinite path of directed open edges $\{
(x_0=0,x_{1})=e_{1},(x_{1},x_{2})=e_{2},\ldots,e_{k},\ldots\}$, that
is, such that $E_{x_{i},x_{i+1}}$ occurs for $i=0,1,\ldots.$
Suppose $\eta^0_0(0)=N-M$. If $e_{1}$ is open, then $\eta^0_t(0)$
reaches~$N$, migrates to $x_{1}$ and gives birth to $N-M$ individuals
on $x_{1}$ before dying out. Then also $e_{2}$ is open, therefore
starting from $\eta^{x_1}_t(x_1)=N-M$, it reaches $N$, migrates to
$x_{2}$ and gives birth to $N-M$ individuals on $x_{2}$ before dying
out, and so on: this is also true for the process $\eta_t\geq\eta^x_t$
for each $x$; therefore, the existence of an infinite path in the
percolation model implies the existence of an infinite path of
individuals.\looseness=1

We begin with one individual at $x \in\mathbb{Z}^d$. For each $\phi
_A<\infty$, with positive probability $\eta^0_t(0)$ reaches $N-M$
before $0$ in a finite time, and we can start our construction.

In order to prove that the existence of an infinite path in percolation
model has positive probability if $\mathbb{P}(E_{x,y})$ is large
enough, one can follow~\cite{cfkuulasmaa}, Theorem 3.2, and compare the
process to a a site percolation model. \textit{Here we need $d\geq2$};
otherwise the construction does not work. The idea consists of making a
comparison with an oriented site percolation model on the square
lattice with both edges from a site open with a given probability~$\pi
$, which can be taken as large as we want by taking $N$ large. Since
for such a model percolation occurs if $\pi$ is large enough, \cite
{cfkuulasmaa}, there is survival with positive probability.

If $\eta_0\in\OmN$, then the upper invariant measure $\bar{\nu}$,
which exists by attractiveness, is not concentrated on the Dirac
measure $\delta_{\underline{0}}$, and the claim follows.

(ii) The proof is similar to that of Theorem \ref
{phasetransition1} [Step (ii)], so we skip it.

\subsection{Model IV}
\subsubsection{\texorpdfstring{Proof of Lemma \protect\ref{existence}}{Proof of Lemma 6.1}}
The process is a particular case of the reaction-diffusion process
introduced in \cite{cfChenbook}, Section 13.2: by following the same
notation, the reaction part of the formal generator (\ref{generator}) is
\[
\mathcal{L}_r f(\eta)=\sum_{x\in\mathbb{Z}^d}\sum_{k\neq0}q_x\bigl(\eta
(x),\eta(x)+k\bigr)[f(S^k_x \eta)-f(\eta)]
\]
with $q_x(\eta(x),\eta(x)+k)=\eta(x)\I_{\{k=1\}}+\phi\eta(x)\I_{\{
k=-1\}}$. The diffusion part is
\begin{eqnarray*}
\mathcal{L}_d f(\eta)&=&\sum_{x\in\mathbb{Z}^d}\sum_{y\sim x}\frac
{\lambda}{2d} \I_{\{\eta(x)\geq N, \eta(y)<N\}}[f(S^{-1,1}_{x,y} \eta
)-f(\eta)]\\
&\leq& \sum_{x\in\mathbb{Z}^d}\sum_{y\sim x}\frac{\lambda}{2d} \I_{\{
\eta(x)\geq N\}}[f(S^{-1,1}_{x,y} \eta)-f(\eta)]\\
&=:& \sum_{x,y\in\mathbb{Z}^d}p(x,y)\lambda c_x(\eta
(x))[f(S^{-1,1}_{x,y} \eta)-f(\eta)],
\end{eqnarray*}
where $p(x,y)=\I_{\{y \sim x\}}(2d)^{-1}$ and $c_x(\eta(x))=\I_{\{\eta
(x)\geq N\}}$.

Since the maximal number of particles involved in a transition is
finite, and the birth and death rates grow linearly, the hypotheses of
\cite{cfChenbook}, Theorems~13.17 and~13.19, are satisfied; hence
existence and uniqueness of this process follow.


\subsubsection{\texorpdfstring{Proof of Theorem \protect\ref{phasetransition4}}{Proof of Theorem 6.1}}

(i) First of all we prove that there is stochastic order between
Model $I$ and Model $\mathit{IV}$. We consider the Model $I$ as a~process
constructed on $\Omega=\mathbb{Z}^{\mathbb{Z}^d}$ with birth rates null
if the number of particles in a site is larger or equal to $N$.\vadjust{\goodbreak}
\begin{lemma}
Let $\xi_{t}=\xi_{t}(\phi',\lambda')$ be a process defined by (\ref
{generator}) with rates~gi\-ven by (\ref{rates1}), that is, a Model
I-type process. Let $\eta_{t}=\eta_{t}(\phi,\widetilde{\phi},\lambda)$\vspace*{1pt}
be a~Model~IV-type process. If $\phi=\phi'$, $\lambda=\lambda'$ and $\xi
_0(x)\leq N$ for each $x \in\mathbb{Z}^d$, then $(\eta_{t})_{t\geq0}$
is stochastically larger than $(\xi_{t})_{t\geq0}$.
\end{lemma}
\begin{pf}
Let $(\widetilde{P}^\cdot_{\cdot,\cdot}, \widetilde{\Gamma}^\cdot_{\cdot
,\cdot})$ and $(P^\cdot_{\cdot,\cdot}, \Gamma^\cdot_{\cdot,\cdot})$ be
respectively the transition rates of $(\xi_{t})_{t\geq0}$ and $(\eta
_{t})_{t\geq0}$.
Note that an increase of particles in a site $x$ with $\xi_t(x)=N$ is
not possible; therefore $\xi_t(x)\leq N$ for each $x\in\mathbb{Z}^d$
and $t\geq0$.

We check conditions in Remark \ref{remM1}. Given $\xi(x)\leq\eta(x)$,
$\xi(y)=\eta(y)$,
\begin{eqnarray*}
\widetilde{P}^1_{\xi(y)}+\widetilde{\Gamma}^1_{\xi(x),\xi(y)}&=&\xi(y)\I
_{\{\xi(y)\leq N-1\}}+\lambda\I_{\{\xi(x)=N, \xi(y)<N\}}\\
&\leq&\eta
(y)+\lambda\I_{\{\eta(x)\geq N,\eta(y)<N\}}\\
&=& {P}^1_{\eta(y)}+\Gamma^1_{\eta(x),\eta(y)},\\
\widetilde{P}^1_{\xi(y)}&=&\xi(y)\I_{\{\xi(y)\leq N-1\}}\leq\eta(y)=
{P}^1_{\eta(y)},
\end{eqnarray*}
and conditions (\ref{C+1})--(\ref{C+10}) are satisfied.

If $\xi(x)=\eta(x)$ [which is possible only if $\eta(x)\leq N$], $\xi
(y)\leq\eta(y)$,
\begin{eqnarray*}
\widetilde{P}^{-1}_{\xi(x)}+\widetilde{\Gamma}^1_{\xi(x),\xi(y)}&=&\phi
\xi(x)+\lambda\I_{\{\xi(x)=N,\xi(y)<N\}}\geq\phi\eta(x)+\lambda\I
_{\{\eta(x)\geq N,\eta(y)<N\}}\\
&=& {P}^{-1}_{\eta(x)}+\Gamma^1_{\eta(x),\eta(y)},\\
\widetilde{P}^{-1}_{\xi(x)}&=&\phi\xi(x)\geq\phi\eta(x)=
{P}^{-1}_{\eta(x)},
\end{eqnarray*}
so that conditions (\ref{C-1})--(\ref{C-10}) hold.
\end{pf}

Therefore by Theorem \ref{phasetransition1} there exists $\phi_c(\lambda
,N)$ such that if $\phi<\phi_c(\lambda,N)$ there is a positive
probability of survival for Model I, and hence for Model IV. By taking
$\phi>1$ one proves as in Model I [Step (ii) in proof of Theorem~\ref
{phasetransition1}], that the process dies out: the existence of the
critical parameter $\phi_c$ follows from monotonicity with respect to
$\phi$.

(ii) We skip this step, since as in Step (i), stochastic
order and Theorem~\ref{phasetransition3} induce survival of the
process.

We prove that even if the process survives, the expected value on each
site is finite. Let $\eta^N_0(x)\geq N$ for each $x \in\mathbb{Z}^d$,
and let $(\eta^{N}_{t})_{t\geq0}$ be a process with $N$ immortal
particles per site, that is, with transition rates
\begin{eqnarray*}
\eta_t^N(x)&\to& \eta_t^N(x)+1 \qquad\mbox{at rate }\eta_t^N(x),\\
\eta_t^N(x)&\to& \eta^N(x)-1 \qquad\mbox{at rate }\widetilde{\phi}\eta
_t^N(x)\I_{\{\eta_t^N(x)>N\}}.
\end{eqnarray*}
We define $\zeta_t(x):=\eta^N_{t}(x)-N$ for each $x \in\mathbb{Z}^d$,
the birth and death process on $\mathbb{N}$ with birth rate $N+\zeta
_{t}(x)$ and death rate $\widetilde{\phi}(N+\zeta_{t}(x))\I_{\{\zeta
_{t}(x)>0\}}$. Thus
\[
\frac{d}{dt}\mathbb{E}(\zeta_{t}(x))= \mathbb{E}\bigl(\zeta
_{t}(x)+N\bigr)-\widetilde{\phi}\mathbb{E}\bigl(\zeta_{t}(x)+N\bigr)\I_{\{\zeta
_t(x)>0\}}\leq N-(\widetilde{\phi}-1)\mathbb{E}(\zeta_{t}(x))
\]
which implies
\[
\mathbb{E}(\zeta_{t}(x))\leq\mathbb{E}(\zeta_{0}(x))+N/(\widetilde{\phi}-1).
\]
Therefore if $\zeta_{0}(x)\leq n$, there exists $c=c(n,N,\widetilde{\phi
})$ such that $\mathbb{E}(\zeta_{t}(x))\leq c$ for each $t\geq0$ and
$x$. The claim follows by taking $C=c+N$.

\section*{Acknowledgments} I am grateful to Ellen Saada, the
French supervisor of my Ph.D. thesis, which was done in joint tutorage
between the LMRS Universit\'e de Rouen and the Universit\`a di Milano
Bicocca. I thank Rinaldo Schinazi and two anonymous referees for very
useful suggestions which helped me to improve the work. I thank
Institut Henri Poincar\'e, Centre Emile Borel for hospitality during
the semester ``Interacting Particle Systems, Statistical Mechanics and
Probability Theory,'' where part of this work was done and Fondation
Sciences Math\'ematiques de Paris for financial support during the
stay. I acknowledge Laboratoire MAP5, Universit\'e Paris Descartes for
hospitality.



\printaddresses

\end{document}